\documentclass[reqno,12pt]{amsart}
\usepackage{latexsym,epsfig,amssymb,amsmath,amsthm,color,url}
\usepackage{hyperref}
\usepackage{enumitem,pifont}

\allowdisplaybreaks 
 \numberwithin{equation}{section}
\newtheorem{thm}{Theorem}[section]
\newtheorem{lem}[thm]{Lemma}

\theoremstyle{definition}
\newtheorem{defn}[thm]{Definition}

\theoremstyle{remark}
\newtheorem{rem}[thm]{Remark}

\DeclareMathOperator{\no}{NO}
\DeclareMathOperator{\yes}{YES}

\DeclareMathOperator{\pr}{\mathbf{P}}
\DeclareMathOperator{\an}{an}

\DeclareMathOperator{\minimum}{Min}
\DeclareMathOperator{\DEHR}{DEHR}
\DeclareMathOperator{\EHR}{EHR}

\DeclareMathOperator{\variance}{Var}
\DeclareMathOperator{\Sup}{SUP}
\DeclareMathOperator{\pic}{PIC}
\DeclareMathOperator{\many}{MANY}
\DeclareMathOperator{\E}{\mathbf{E}}
\DeclareMathOperator{\bu}{BUSHY}

\title{First order theory on $G(n, c/n)$}
\author{Moumanti Podder}
\thanks{The author was partially supported by NSF CAREER under grant CCF:AF-1553354}
\address{Moumanti Podder, \ Georgia Institute of Technology, \ School of Mathematics, \ 686 Cherry St NW, Atlanta, GA 30313, United States.}
\email{mpodder3@math.gatech.edu.}

\begin{document}
\bibliographystyle{plainnat}
\nocite{*}

\begin{abstract}A well-known result of Shelah and Spencer tells us that the almost sure theory for first order language on the random graph sequence $\left\{G(n, cn^{-1})\right\}$ is not complete. This paper proposes and proves what the complete set of completions of the almost sure theory for $\left\{G(n, c n^{-1})\right\}$ should be. The almost sure theory $T$ consists of two sentence groups: the first states that all the components are trees or unicyclic components, and the second states that, given any $k \in \mathbb{N}$ and any finite tree $t$, there are at least $k$ components isomorphic to $t$. We define a $k$-completion of $T$ to be a first order property $A$, such that if $T + A$ holds for a graph, we can fully describe the first order sentences of quantifier depth $\leq k$ that hold for that graph. We show that a $k$-completion $A$ specifies the numbers, up to ``cutoff" $k$, of the (finitely many) unicyclic component types of given parameters (that only depend on $k$) that the graph contains. A complete set of $k$-completions is then the finite collection of all possible $k$-completions.


\end{abstract}



\maketitle

\section{Introduction}
\sloppy A well-known result of Shelah and Spencer \big(\cite[Theorem 6]{05}\big) states that the edge probability $p(n) = n^{-\alpha}$ in $G\big(n, p(n)\big)$ satisfies the zero-one law if and only if $\alpha$ is irrational. This means that when $\alpha$ is irrational, for every first order (FO) graph property $A$, the probability that $G\big(n, n^{-\alpha}\big)$ satisfies $A$ goes either to $0$ or to $1$ as $n$ goes to $\infty$. On the other hand, when $\alpha \in (0, 1]$ is rational, there exists some FO property whose probability under the measure induced by $G\big(n, n^{-\alpha}\big)$ approaches a limit in $(0, 1)$. As a result, the almost sure theory for FO logic with respect to $n^{-\alpha}$ is not complete when $\alpha$ is rational \big(these notions are explained in detail in Subsection \ref{FO_theory_model}\big). This work establishes the completion of the almost sure theory for FO language with respect to $p(n) = c n^{-1}$. The exact value of $c$ will be inconsequential, a fact that becomes evident from the main result. 
\par The two main results, Theorems~\ref{main_1} and \ref{main_2}, involve several notions that need to be introduced before the formal statements of the theorems can be explained. Before stating the theorems, we give a rough overview of their contents. The almost sure theory for FO sentences in $G\big(n, c n^{-1}\big)$ is described in Theorem~\ref{main_1}, and it comprises two groups of sentences. The first group asserts that any component of $G\big(n, c n^{-1}\big)$ will  be a tree or a unicyclic graph, i.e.\ there will be no component that is bicyclic or of higher graph-complexity. The second group, roughly speaking, comprises the following sentences: given any positive integer $k$, and any finite tree $t$, there will be at least $k$ tree-components that are of the same \emph{tree-type} as $t$ \big(see Definition \ref{tree_type} for tree-types\big). The completion of the almost sure theory for $G\big(n, c n^{-1}\big)$ is given in Theorem~\ref{main_2}, and a sentence $A$ in this group can be roughly described as follows. Given any positive integers $k$, $s$ and $m$, and any unicyclic graph $U$ with the length of its cycle $s$ and the trees originating from the vertices on its cycle of depth at most $m$, $A$ specifies the number (counted up to cutoff $k$) of unicyclic components that have the same \emph{cycle-type} as $U$ \big(see Definition \ref{unicyclic_type} for cycle-types\big). Let $\Sigma_{m, k}$ denote the set of all tree-types of depth $m$ and cutoff $k$, and $\Gamma_{s, m, k}$ the set of all cycle-types with cycle length $s$, maximum depth of trees $m$ and cutoff $k$. The formal statement of the theorem is as follows \big(throughout this paper, we shall denote by $\mathbb{N}_{s}$ the set of all non-negative integers that are at least $s$, for any non-negative integer $s$, and $\mathbb{N}_{1}$ is simply $\mathbb{N}$\big):

\begin{thm}\label{main_1}
Consider the theory $T$ consisting of the following sentences:
\begin{enumerate}
\item \label{1} For all $\ell \in \mathbb{N}$, the sentence $\no_{\ell}$: there does not exist any subset of $\ell$ vertices with $\ell+1$ edges.
\item \label{2} For each $m \in \mathbb{N}_{0}$, $k \in \mathbb{N}$ and tree-type $\sigma \in \Sigma_{m, k}$, the sentence $\yes_{\sigma}$: there exist at least $k$ tree components with $(m, k)$-type $\sigma$.
\end{enumerate}
Each sentence in $T$ holds almost surely for $p(n) = c n^{-1}$. Moreover, every first order sentence $B$ that holds almost surely for the edge probability sequence $p(n) = c n^{-1}$, can be derived from $T$.
\end{thm}


\begin{thm}\label{main_2}
Consider the countably infinite index set $I$ consisting of all infinite sequences of the form
\begin{equation}\label{I}
\vec{n} = \Big(n_{\gamma}: \gamma \in \Gamma_{s, m, k},\ s \in \mathbb{N}_{3},\ m \in \mathbb{N}_{0},\ k \in \mathbb{N},\Big)
\end{equation}
where each $n_{\gamma} \in \{0, 1, \ldots, k\}$, and which are consistent (see \eqref{consistency_cond} for the definition of consistent sequences). For every $\vec{n} \in I$, consider the property $\mathcal{A}_{\vec{n}}$: for each $s \in \mathbb{N}_{3}$, $k \mathbb{N}$, $m \mathbb{N}_{0}$, and $\gamma \in \Gamma_{s, m, k}$, there exist exactly $n_{\gamma}$ many unicyclic components with $(s, m, k)$-type $\gamma$ when $n_{\gamma} < k$, and there exist at least $k$ unicyclic components with $(s, m, k)$-type $\gamma$ when $n_{\gamma} = k$. Then the family $\big\{\mathcal{A}_{\vec{n}}: \vec{n} \in I\big\}$ is a complete set of completions for $T$.
\end{thm}

\par We shall actually state and prove a stronger version of Theorem~\ref{main_2}: one that also shows that, given any positive integer $k$, the number of completions of the almost sure theory for FO sentences with quantifier depth at most $k$ is \emph{finite}. This version is given in Theorem~\ref{main_3}, where we only care about the counts of unicyclic components that have cycle length at most $2 \cdot 3^{k+3}$ and maximum depth of trees at most $3^{k+3}$.

\begin{thm}\label{main_3}
Fix a positive integer $k$. Consider the \emph{finite} index set $I_{k}$ consisting of all sequences of the form
\begin{equation}\label{I_truncated}
\vec{n} = \Big(n_{\gamma}: \gamma \in \Gamma_{s, m, k},\ s \in \{3, \ldots, 2 \cdot 3^{k+3}\},\ m \in \{0, \ldots, 3^{k+3}\}\Big)
\end{equation}
where $n_{\gamma} \in \{0, 1, \ldots, k\}$, and the sequence is consistent for every integer $3 \leq s \leq 3^{k+2}$, i.e.\ satisfies \eqref{consistency_cond_truncated} with $M_{1} = 2 \cdot 3^{k+3}$ and $M_{2} = 3^{k+3}$. For every $\vec{n} \in I_{k}$, we consider the property $\mathcal{A}_{\vec{n}}$: for all integers $3 \leq s \leq 2 \cdot 3^{k+3}$, $0 \leq m \leq 3^{k+3}$ and $\gamma \in \Gamma_{s, m, k}$, there exist exactly $n_{\gamma}$ many unicyclic components with $(s, m, k)$-type $\gamma$ when $n_{\gamma} < k$, and there exist at least $k$ unicyclic components with $(s, m, k)$-type $\gamma$ when $n_{\gamma} = k$. Then the family $\big\{\mathcal{A}_{\vec{n}}: \vec{n} \in I_{k}\big\}$ is a complete set of $k$-completions for $T$.
\end{thm}

\subsection{First order language, theories and models}\label{FO_theory_model}
For this entire subsection, an excellent source to refer to is \cite[Chapters~1 and 3]{01}. First order language on graphs comprises sentences that capture \emph{local} properties of graphs. Formally, this language consists of the following components:
\begin{enumerate}
\item equality $(=)$ of vertices and adjacency $(\sim)$ of vertices;
\item variable symbols that are vertices in the graph, denoted by $x, y, z, \ldots$ etc.;
\item usual Boolean connectives such as $\vee, \wedge, \neg, \implies, \Leftrightarrow$ etc.;
\item quantifications: existential $(\exists)$ and universal $(\forall)$ that are \emph{only allowed over vertices}.
\end{enumerate}
A first order property of a graph is expressible as a sentence of finite length in this language. The \emph{quantifier depth} of an FO property is the minimum number of nested quantifiers needed to write it. Henceforth, prior to proving any result, we shall fix an arbitrary positive integer $k$, and all FO properties we then consider will have quantifier depth at most $k$.  
\par An FO property $A$ holds \emph{almost surely} with respect to the edge probability function $p(n)$ if
\begin{equation}\label{a.s. prop}
\lim_{n \rightarrow \infty} \pr\left[G\left(n, p_{n}\right) \models A\right] = 1,
\end{equation}
where $\models$ implies that the property $A$ holds. We say that $A$ holds \emph{almost never} if $\neg A$ holds almost surely. A \emph{theory} refers to a collection of sentences that are closed under logical inference in the FO language. We call a theory \emph{consistent} if it does not contain a contradiction, i.e.\ both a sentence and its negation are not present. We define the \emph{almost sure theory} $T$ relative to $p(n)$ to be the set of all FO properties that hold almost surely. We define a theory $T$ to be \emph{complete} if for every FO property $A$, either $A$ or $\neg A$ is in the theory. We define a theory $T$ to be \emph{$k$-complete}, for a given positive integer $k$, if for every FO sentence $A$ with quantifier depth at most $k$, either $A$ or $\neg A$ is in the theory. It is straightforward to see that the almost sure theory with respect to $p(n)$ is complete if and only if $p(n)$ satisfies the zero-one law. 
\par When $p(n)$ does not satisfy the zero-one law, in some cases it is still possible to ``nicely describe" a set of sentences that are not in the almost sure theory, and when appended to the almost sure theory, makes it complete. This is defined as follows \big(see \cite[Subsection 3.6.1]{01}\big). Suppose $T$ is the almost sure theory with respect to $p(n)$, and $I$ is a countable index set. Let $\mathcal{A}_{i}, i \in I$, be a set of sentences that are not in $T$. Then the family $\left\{\mathcal{A}_{i}: i \in I\right\}$ is said to be a \emph{complete set of completions} for $T$ if the following hold:
\begin{enumerate}[label=(\alph*)]
\item \label{a} for every $i \in I$, $T + \mathcal{A}_{i}$ is complete,
\item \label{b} for all $i, j \in I$, $T \models \neg\big(\mathcal{A}_{i} \wedge \mathcal{A}_{j}\big)$,
\item \label{c} for every $i \in I$, the limit $\displaystyle \lim_{n \rightarrow \infty} P\big[G(n, p(n)) \models \mathcal{A}_{i}\big]$ exists and the sum of these limits over $i \in I$ is $1$.
\end{enumerate}
We call $\left\{\mathcal{A}_{i}: i \in I\right\}$ a \emph{complete set of $k$-completions} for $T$ if for every $i \in I$, $T + \mathcal{A}_{i}$ is $k$-complete, and both \ref{b} and \ref{c} hold. 
\par A \emph{model} $G$ of a theory $T$ is a graph that satisfies all the sentences that are in $T$. G\"{o}del's completeness theorem \big(see \cite{04}\big) states that any consistent theory must have a finite or countable model. In our case, the model will be countably infinite. A theory is said to be \emph{$\aleph_{0}$-categorical} if it has precisely one countable model up to isomorphism. We call two graphs $G_{1}, G_{2}$ \emph{elementarily equivalent} if they satisfy the same set of FO properties. 
\par We prove elementary equivalence via the well-known tool of \emph{Ehrenfeucht games}. The Eherenfeucht games \big(see \cite[Chapter~2]{01}\big) serve as a bridge between mathematical logical properties and their structural descriptions on graphs (or even more general settings). Since this paper focuses on FO logic, we shall only be concerned with the standard \emph{pebble-move Ehrenfeucht games}. This game is played between two players, Spoiler and Duplicator, on two given graphs $G_{1}$ and $G_{2}$, for a given number of rounds $k$. In each round, Spoiler selects \emph{either} of the two graphs, and selects a vertex from that graph; in reply, Duplicator, in the same round, selects a vertex from the \emph{other} graph. Suppose $x_{i}$ is the vertex selected from $G_{1}$ and $y_{i}$ that from $G_{2}$ in round $i$, for $1 \leq i \leq k$. Duplicator wins, assuming optimal play by both players, if for all $i, j \in \{1, \ldots, k\}$, 
\begin{enumerate}[label={(EHR \arabic*)},leftmargin=*]
\item \label{EHR 1} $x_{i} = x_{j} \Leftrightarrow y_{i} = y_{j}$,
\item \label{EHR 2} $x_{i} \sim x_{j} \Leftrightarrow y_{i} \sim y_{j}$.
\end{enumerate}
\cite[Theorems~2.3.1 and 2.3.2]{01} give us the connection between FO logic on graphs and pebble-move Ehrenfeucht games. Duplicator wins the $k$-round Ehrenfeucht game $\EHR[G_{1}, G_{2}, k]$ if and only if $G_{1}$ and $G_{2}$ satisfy exactly the same FO sentences of quantifier depth at most $k$. To show that two countable models $G_{1}$ and $G_{2}$ for a theory in the FO language are elementarily equivalent, it therefore suffices to show that Duplicator wins $\EHR[G_{1}, G_{2}, k]$ for all $k$.

\subsection{Notations}\label{notation}
The root of a tree $T$ is generally denoted by $\phi$. For any vertex $v$ in $T$ that is not the root, let $\pi(v)$ denote its parent. For any vertex $v$ of $T$, let $d(v)$ denote its depth in $T$ (where we set $d(\phi) = 0$). Let $d(T)$ denote the maximum depth of $T$, i.e.\ $\sup\{d(v): v \in T\}$. Let $T(v)$ denote the subtree of $T$ consisting of $v$ and all its descendants in $T$. When $v$ is a child of the root $\phi$, we call $v$ a \emph{rootchild} and $T(v)$ a \emph{principal branch} of the tree. Let $T|_{n}$ denote the truncation of $T$ consisting of vertices at depth at most $n$. 
\par Consider a unicyclic graph $U$, where the vertices on the cycle are, say in anticlockwise order, named $v_{1}, \ldots, v_{s}$. We call these vertices \emph{cycle-vertices} of $U$. The exact order in which we enumerate them (as long as we enumerate them consecutively along the cycle) will cease to matter, as we shall see from Remark~\ref{cycle_convention}. Let $T_{i}$ be the tree rooted at $v_{i}$ for $1 \leq i \leq s$. We write $U = \left(v_{1}, \ldots, v_{s}; T_{1}, \ldots, T_{s}\right)$. For a vertex $v$ in $U$, if $v \in T_{i}$, then we call $v_{i}$ the \emph{cycle-ancestor} of $v$, and denote it by $\an(v)$ ($v_{i}$ is its own cycle-ancestor, for every $1 \leq i \leq s$). In other words, $\an(v)$ is the cycle-vertex of $U$ which is closest to $v$. For a vertex $v$ with $\an(v) = v_{i}$ for some $1 \leq i \leq s$, we define the depth $d(v)$ to be its depth with respect to the tree $T_{i}$. The truncation $U|_{n}$ of $U$ consists of all vertices with depth at most $n$, i.e.\ $U|_{n} = \left(v_{1}, \ldots, v_{s}; T_{1}|_{n}, \ldots, T_{s}|_{n}\right)$. By \emph{maximum depth} of $U$ we mean $\max\left\{d(T_{i}): 1 \leq i \leq s\right\}$, which could be infinite.
\par The graph distance is denoted by $\rho$ (thus $u \sim v \Leftrightarrow \rho(u, v) = 1$).

\section{Types of trees and unicyclic graphs}\label{types}
Here we shall define \emph{types} of trees and unicyclic graphs with respect to certain parameters. Once these parameters are fixed, there are only a \emph{finite} number of types into which all trees or all unicyclic graphs get classified. 
\begin{defn}[Tree types] \label{tree_type}
This definition requires two parameters: a \emph{cutoff} $k$ and a \emph{maximum depth} $m$ up to which a tree is considered. For a given tree $T$, its $(m, k)$-type is defined to be the same as the $(m, k)$-type of $T|_{m}$. Hence, it is enough to define $(m, k)$-types for trees of depth at most $m$.
\par We define the $(m, k)$-type of a tree recursively, where the recursion happens on depth. The only $(0, k)$-type is the root itself. Suppose we have defined all possible $(m-1, k)$-types, and let $\Sigma_{m-1,k}$ denote the set of all these types. Given a tree $T$ of depth at most $m$, let $n_{\sigma}$ be the number of principal branches with $(m-1, k)$-type $\sigma$, for every $\sigma \in \Sigma_{m-1, k}$. Then the $(m, k)$-type of $T$ is given by the vector
\begin{equation}
\big(n_{\sigma} \wedge k: \sigma \in \Sigma_{m-1, k}\big).
\end{equation}
\end{defn}
Notice that we are truncating the count for each $(m-1,k)$-type at $k$. This is the reason why we have only finitely many $(m, k)$-types, i.e.\ $\Sigma_{m, k}$ is finite.

\par Before we define types for unicyclic graphs, we set down a convention that will help us define types so that they are the same up to dihedral automorphisms. Consider a totally ordered finite set $(S, \prec)$. A \emph{word} from $S$ is a finite ordered sequence of elements from $S$. Consider a word $\left(a_{1}, \ldots, a_{s}\right)$ from $S$. Let $D_{s}$ be the dihedral permutation group of order $2s$. For any word $(a_{1}, \ldots, a_{s})$ from $S$, we set 
\begin{equation}\label{dihedral}
D_{s}\big(\left(a_{1}, a_{2}, \ldots, a_{s}\right)\big) = \big\{\left(a_{\pi(1)}, a_{\pi(2)}, \ldots, a_{\pi(s)}\right): \pi \in D_{s}\big\}.
\end{equation}
\par Let $x \preceq y$ indicate that either $x \prec y$ or $x = y$. We now define a total ordering $\prec_{S}$ on $S^{s}$ for every $s \in \mathbb{N}$. For two words $\left(x_{1}, \ldots, x_{s}\right)$ and $\left(y_{1}, \ldots, y_{s}\right)$, if $x_{t} \preceq y_{t}$ for all $1 \leq t \leq s$, with at least one $t_{0}$ such that $x_{t_{0}} \prec y_{t_{0}}$, we define $\left(x_{1}, \ldots, x_{s}\right) \prec_{s} \left(y_{1}, \ldots, y_{s}\right)$. We denote by $\minimum\big(\left(a_{1}, \ldots, a_{s}\right)\big)$ the \emph{minimal element} in $D_{s}\big(\left(a_{1}, a_{2}, \ldots, a_{s}\right)\big)$ under the total ordering $\prec_{s}$. 

\par By Definition \ref{tree_type}, $\Sigma_{m-1, k} \subseteq \Sigma_{m, k}$. We choose a total ordering $\prec_{m}$ on $\Sigma_{m, k}$ such that $\prec_{m}$ restricted to $\Sigma_{m-1, k}$ agrees with $\prec_{m-1}$, and for \emph{all} $\sigma_{1} \in \Sigma_{m-1, k}$ and $\sigma_{2} \in \Sigma_{m, k} \setminus \Sigma_{m-1, k}$, we have $\sigma_{1} \prec_{m} \sigma_{2}$. Such a total ordering is non-unique, but we fix an arbitrary choice for the rest of this paper.

\begin{defn}[Unicyclic graph types]\label{unicyclic_type}
This definition requires three parameters: cutoff $k$, maximum depth $m$ and cycle-length $s$; we also need the total ordering $\prec_{s}$ fixed above. Given uncyclic graph $U = \left(v_{1}, \ldots, v_{s}; T_{1}, \ldots, T_{s}\right)$, let $\sigma_{i}$ be the $(m, k)$-type of $T_{i}$ for $1 \leq i \leq s$. Consider the length-$s$ word $\left(\sigma_{1}, \ldots, \sigma_{s}\right)$ from $\Sigma_{m, k}$. Then the $(s, m, k)$-type of $U$ is given by $\minimum\big(\left(\sigma_{1}, \ldots, \sigma_{s}\right)\big)$ defined under $\prec_{s}$. We shall denote the set of all possible $(s, m, k)$-types by $\Gamma_{s, m, k}$. 
\end{defn}

\begin{rem}\label{cycle_convention}
Henceforth, we shall always maintain the following convention while describing any unicyclic graph $U$. Suppose $U$ has cycle length $s$, and fix any integers $m \in \mathbb{N}_{0}$ and $k \in \mathbb{N}$. We enumerate the cycle-vertices of $U$ (again, consecutively along the cycle) as $v_{1}, \ldots, v_{s}$ such that, if $\sigma_{i}$ is the $(m, k)$-type of the tree $T_{i}$ rooted at $v_{i}$, for $1 \leq i \leq s$, then $(\sigma_{1}, \ldots, \sigma_{s})$ is the $(s, m, k)$-type of $U$ (i.e.\ $(\sigma_{1}, \ldots, \sigma_{m})$ is the minimum, under $\prec_{s}$, out of the dihedral permutations of the cycle-vertices of $U$).
\end{rem}

We now define what it means for a sequence of the form given in \eqref{I}, or in \eqref{I_truncated}, to be consistent. Consider arbitrary but fixed $s \in \mathbb{N}_{3}$ and $k \in \mathbb{N}$. For any $m \in \mathbb{N}_{0}$, $\gamma \in \Gamma_{s, m, k}$ and $i \geq m$, set
\begin{equation}\label{sup}
\Sup(\gamma) = \left\{\sigma \in \Gamma_{s, m+1, k}: \text{ the } (s, m, k)-\text{type of } \sigma \text{ is } \gamma\right\}.
\end{equation}
The type $\gamma$ itself belongs to $\Sup(\gamma)$ since $\Gamma_{s, m, k} \subseteq \Gamma_{s, m+1, k}$. We now define the consistency condition required in \eqref{I} as follows:
\begin{equation}\label{consistency_cond}
n_{\gamma} = \left\{\sum_{\sigma \in \Sup(\gamma)} n_{\sigma}\right\} \wedge k \text{ for all } m \in \mathbb{N}_{0},\ s \in \mathbb{N}_{3},\ \gamma \in \Gamma_{s, m, k}.
\end{equation} 
Fix positive integers $M_{1}$ and $M_{2}$, and vector $\vec{n} = \left(n_{\gamma}: \gamma \in \Gamma_{s, m, k}, 3 \leq s \leq M_{1}, 0 \leq m \leq M_{2}\right)$, where each $n_{\gamma} \in \{0, 1, \ldots, k\}$. We say that this vector is consistent if we have
\begin{equation}\label{consistency_cond_truncated}
n_{\gamma} = \left\{\sum_{\sigma \in \Sup(\gamma)} n_{\sigma}\right\} \wedge k \text{ for all } 0 \leq m \leq M_{2},\ 3 \leq s \leq M_{1},\ \gamma \in \Gamma_{s, m, k} .
\end{equation}

\section{The almost sure theory with respect to $p(n) = c n^{-1}$}\label{schema}
Having defined types, we are in a position to explicitly describe the groups of sentences given in Theorems~\ref{main_1} and \ref{main_2}. We show here that each sentence in groups \ref{1} and \ref{2} holds almost surely with respect to $p(n) = c n^{-1}$. 

\par For $\ell \in \mathbb{N}$, recall that $\no_{\ell}$ is the property that there exists no subset of $\ell$ vertices with $\ell+1$ edges. For any subset $S$ of $\ell$ vertices in $G(n, cn^{-1})$, let $X_{S}$ be the indicator random variable for the event that $S$ has $\ell+1$ edges. Then 
$$E[X_{S}] = {{\ell \choose 2} \choose \ell+1} \left\{c n^{-1}\right\}^{\ell+1} = \Theta\left(n^{-\ell-1}\right).$$
There are ${n \choose \ell} = \Theta\left(n^{\ell}\right)$ many such subsets $S$, hence the expected number of subsets with $\ell$ vertices and $\ell+1$ edges, is $\Theta\left(n^{\ell} \cdot n^{-\ell-1}\right) = \Theta\left(n^{-1}\right)$. A direct application of Chebychev's inequality now shows that $\no_{\ell}$ holds almost surely.


\par Now fix $m \in \mathbb{N}_{0}$ and $k \in \mathbb{N}$. We prove a result stronger than $\yes_{\sigma}$, for any $\sigma \in \Sigma_{m, k}$. 
\begin{lem}\label{stronger_yes}
Given any finite tree $t$ and any positive integer $k$, a countable model of the theory $T$, as given in Theorem~\ref{main_1}, will contain at least $k$ tree components each of which is isomorphic to $t$.
\end{lem}
\begin{proof}
Fix \emph{any} finite tree $t$. Suppose $t$ has $\ell$ vertices (hence $\ell-1$ many edges), and $a$ automorphisms. Consider any subset $S$ of $\ell$ vertices. Let $Y_{S}$ be indicator for the event that the subgraph on $S$ is a component by itself and isomorphic to $t$. The number of distinct graphs on $S$ that are isomorphic to $t$, up to automorphism, is $\ell! a^{-1}$. For $S$ to be a component by itself, there can be no edge between $S$ and $S^{c}$. Hence
\begin{align}\label{a.s. eq 1}
E\left[Y_{S}\right] &= \frac{\ell!}{a} \left(\frac{c}{n}\right)^{\ell - 1} \left(1-\frac{c}{n}\right)^{\ell(n - \ell) + {\ell \choose 2} - \ell + 1} \nonumber\\
&= \Theta\left\{n^{-\ell+1} \left(1 - \frac{c}{n}\right)^{\ell n}\right\} = \Theta\left\{n^{-\ell+1} e^{-\ell c}\right\} = \Theta\left(n^{-\ell+1}\right).
\end{align} 
There are ${n \choose \ell} = \Theta\left(n^{\ell}\right)$ many choices for the subset $S$. Hence the expected number of such tree components is $\Theta\left(n^{\ell} \cdot n^{-\ell+1}\right) = \Theta(n)$.
Crucially, note that if $S_{1}$ and $S_{2}$ are two subsets of $\ell$ vertices such that they overlap in at least one vertex, then they cannot simultaneously be components isomorphic to $t$, i.e.\ $Y_{S_{1}}$ and $Y_{S_{2}}$ cannot simultaneously be $1$. When $S_{1}$ and $S_{2}$ do not overlap, $Y_{S_{1}}$ and $Y_{S_{2}}$ are independent. We thus get (using \eqref{a.s. eq 1}):
\begin{align}\label{a.s. eq 2}
\variance\left[\sum_{S \subseteq G(n, c n^{-1}), |S| = \ell} Y_{S}\right] &= \sum_{\substack{S \subseteq G(n, c n^{-1})\\|S| = \ell}} \left\{E\left[Y_{S}\right] - E^{2}\left[Y_{S}\right]\right\} - \sum_{\substack{|S_{1}| = |S_{2}| = \ell\\S_{1} \cap S_{2} \neq \emptyset}} E\left[Y_{S_{1}}\right] E\left[Y_{S_{2}}\right] \nonumber\\
&= \Theta(n) - \sum_{i=1}^{\ell-1} \sum_{\substack{|S_{1}| = |S_{2}| = \ell\\\bigl|S_{1} \cap S_{2}\bigr| = i}} E\left[Y_{S_{1}}\right] E\left[Y_{S_{2}}\right] \nonumber\\
&= \Theta(n) - \sum_{i=1}^{\ell-1} {n \choose i} {n-i \choose \ell - i} {n-\ell \choose \ell-i} \cdot \Theta\left(n^{-\ell+1}\right) \cdot \Theta\left(n^{-\ell+1}\right) = \Theta(n).
\end{align}
We now conclude that 
\begin{equation}
\variance\left[\sum_{S \subseteq G(n, c n^{-1}), |S| = \ell} Y_{S}\right] = o\left\{E^{2}\left[\sum_{S \subseteq G(n, c n^{-1}), |S| = \ell} Y_{S}\right]\right\}.
\end{equation}
From \cite[Corollary~4.3.3]{06}, we conclude that $\sum_{S \subseteq G(n, c n^{-1}), |S| = \ell} Y_{S} \sim E\left[\sum_{S \subseteq G(n, c n^{-1}), |S| = \ell} Y_{S}\right]$ almost always, which means that with probability tending to $1$, $G(n, cn^{-1})$ will contain $\Theta(n)$ isolated copies of $t$. 
\end{proof}
This lemma implies that $\yes_{\sigma}$ holds almost surely, as desired. 

\par Lastly, we have to verify the following: if $A$ is any FO sentence that holds almost surely for the sequence $p(n) = c n^{-1}$, then $A$ is derivable from $T$. We show the proof here conditional on us proving Theorem \ref{main_3}. Once we establish Theorem \ref{main_3}, we know the following, for every fixed $k \in \mathbb{N}$:
\begin{enumerate}
\item \label{conclusion_1} For any $\vec{n} \in I_{k}$, the limiting probability
$$\lim_{n \rightarrow \infty} \pr\left[G(n, c n^{-1}) \models \mathcal{A}_{\vec{n}}\right] > 0.$$
\item \label{conclusion_2} For a fixed $\vec{n} \in I_{k}$, we know the exact set of all FO sentences of quantifier depth $\leq k$ that hold for the theory $T + \mathcal{A}_{\vec{n}}$. That is, any countable model for the theory $T + \mathcal{A}_{\vec{n}}$ satisfies a specific set of FO sentences, of quantifier depth at most $k$, that only depends on $\vec{n}$. 
\item \label{conclusion_3} Any given countable model that satisfies $T$ must satisfy $\mathcal{A}_{\vec{n}}$ for exactly one $\vec{n} \in I_{k}$.
\end{enumerate}
Consider $A$ with quantifier depth $k$. Since $A$ holds almost surely with respect to $p(n) = c n^{-1}$, hence from \ref{conclusion_1} and \ref{conclusion_2} above, we can conclude that for every $\vec{n} \in I_{k}$, any countable model of $T + \mathcal{A}_{\vec{n}}$ satisfies $A$. Consequently, from \ref{conclusion_3}, any countable model of $T$ will satisfy $A$, and this shows that indeed $A$ is derivable from $T$. This completes the verification that $T$, as described in Theorem~\ref{main_1}, is indeed the almost sure theory with respect to $p(n) = c n^{-1}$.

\section{The distance preserving Ehrenfeucht game}\label{DEHR}
The distance preserving Ehrenfeucht game (DEHR) is needed as a local tool, for constructing Duplicator's winning strategy for the pebble-move Ehrenfeucht game on two countable models for $T + \mathcal{A}_{\vec{n}}$, $\vec{n} \in I_{k}$. This is a full information game played between Spoiler and Duplicator, where both players are assumed to play optimally. The players are given two graphs $G_{1}$ and $G_{2}$, the number of rounds $k$, and pairs $(x_{i}, y_{i})$, $1 \leq i \leq \ell$, in $G_{1} \times G_{2}$, where $\ell$ is some non-negative integer. These pairs are known as \emph{designated pairs}, and can be thought of as outcomes of earlier rounds in the game. In particular, when $\ell = 0$, no such pair is given. Let $C = \big\{(x_{i}, y_{i}) : 1 \leq i \leq \ell\big\}$.
\par In each of the $k$ rounds, Spoiler chooses \emph{either} of the two graphs, and then selects a vertex from that graph. In reply, within the same round, Duplicator chooses a vertex from the other graph. Let $x_{i+\ell}$ denote the vertex selected from $G_{1}$, and $y_{i+\ell}$ that from $G_{2}$, in round $i$, for $1 \leq i \leq k$. Duplicator wins this game \big(denoted $\DEHR\big[G_{1}, G_{2}, k, C\big]$\big), if \emph{all} of the following conditions hold: for all $i, j \in \{1, \ldots, k+\ell\}$, 
\begin{enumerate}[label={(DEHR \arabic*)},leftmargin=*]
\item \label{DEHR win 1} $\pi(x_{j}) = x_{i} \Leftrightarrow \pi(y_{j}) = y_{i}$,
\item \label{DEHR win 2} $\rho(x_{i}, x_{j}) = \rho(y_{i}, y_{j})$,
\item \label{DEHR win 3} $x_{i} = x_{j} \Leftrightarrow y_{i} = y_{j}$.
\end{enumerate}

\par Given graphs $G_{1}, G_{2}$, $k \in \mathbb{N}$, and $C = \big\{(x_{i}, y_{i}) : 1 \leq i \leq \ell\big\}$ of designated pairs, we call $C$ \emph{winnable} for $\big\{G_{1}, G_{2}, k\big\}$ if Duplicator wins $\DEHR\big[G_{1}, G_{2}, k, C\big]$. Moreover, given $G_{1}$, $G_{2}$, $k$, a winnable $C$ as above, and any vertex $x$ in $G_{1}$, we define $y$ in $G_{2}$ to be a \emph{corresponding vertex} to $x$ with respect to $\big\{G_{1}, G_{2}, k, C\big\}$, if $C \cup \{(x, y)\}$ is winnable for $\big\{G_{1}, G_{2}, k-1\big\}$. We analogously define a corresponding vertex in $G_{1}$ to every vertex in $G_{2}$ with respect to $\big\{G_{1}, G_{2}, k, C\big\}$. The choice of corresponding vertices need not be unique, but there exists at least one corresponding vertex in $G_{2}$ for every vertex in $G_{1}$, and vice versa, since $C$ is winnable.
\par Note that, when $G_{1}$ and $G_{2}$ are such that Duplicator wins $\DEHR[G_{1}, G_{2}, k]$, given any vertex $x \in G_{1}$, we can find a corresponding vertex $y$ in $G_{2}$ to $x$, and vice versa.

\par We now state a lemma that shows us that Duplicator wins the distance preserving Ehrenfeucht game when it is played on two trees of the same type (where the tree-type is defined with suitable parameters).

\begin{lem}\label{same_tree_type_DEHR}
Let $T_{1}, T_{2}$ be two trees with roots $\phi_{1}, \phi_{2}$, and the same $(m, k)$-type. Then Duplicator wins $\DEHR\big[T_{1}|_{m}, T_{2}|_{m}, k\big]$, with the designated pair $(\phi_{1}, \phi_{2})$.
\end{lem}

\begin{proof}
The proof is via induction on $m$. The case $m = 0$ is immediate. Suppose the claim holds for some $m \geq 0$. Let $T_{1}$ and $T_{2}$ have the same $(m+1, k)$-types. For this proof, we abbreviate $\Sigma_{i, k}$ by $\Sigma_{i}$ for all $i$. For $\sigma \in \Sigma_{m}$, let $S_{\sigma, i}$, for $1 \leq i \leq n^{(1)}_{\sigma}$, be the principal branches of $T_{1}|_{m+1}$, and $T_{\sigma, j}$, for $1 \leq j \leq n^{(2)}_{\sigma}$, those of $T_{2}|_{m+1}$, with $(m, k)$-type $\sigma$. From the definition of types, we have
\begin{equation}\label{same type count}
n^{(1)}_{\sigma} \wedge k = n^{(2)}_{\sigma} \wedge k \text{ for all } \sigma \in \Sigma_{m}.
\end{equation}
By induction hypothesis, we know that, for all $1 \leq i \leq n^{(1)}_{\sigma}$ and $1 \leq j \leq n^{(2)}_{\sigma}$,
\begin{equation}\label{ind hyp same type}
\text{Duplicator wins } \DEHR\Big[S_{\sigma, i}, T_{\sigma, j}, k, \left(\phi^{1}_{\sigma, i}, \phi^{2}_{\sigma, j}\right)\Big], 
\end{equation}
where $\phi^{1}_{\sigma, i}$ is the root of $S_{\sigma, i}$ and $\phi^{2}_{\sigma, j}$ that of $T_{\sigma, j}$. Suppose $s$ rounds of the game have been played, and the vertex selected from $T_{1}|_{m+1}$ in round $i$ is $x_{i}$, and that from $T_{2}|_{m+1}$ is $y_{i}$. Duplicator maintains the following conditions on the configuration $\{(x_{i}, y_{i}): 1 \leq i \leq s\}$, for every $1 \leq s \leq k$: 
\begin{enumerate}[label={(A\arabic*)},leftmargin=*]
\item \label{same type cond 1} $x_{i} = \phi_{1} \Leftrightarrow y_{i} = \phi_{2}$.

\item \label{same type cond 3} For $1 \leq i_{1} < \cdots < i_{r} \leq s$, call $\big\{x_{i_{1}}, \ldots, x_{i_{r}}\big\}$ an \emph{$x$-cluster} up to round $s$, if they belong to a common principal branch, and no other $x_{j}$ selected so far belongs to it. We analogously define a \emph{$y$-cluster}, up to round $s$. Then $\big\{x_{i_{1}}, \ldots, x_{i_{r}}\big\}$ is an $x$-cluster iff $\big\{y_{i_{1}}, \ldots, y_{i_{r}}\big\}$ is a $y$-cluster, and the principal branches they belong to are of the same $(m, k)$-type.

\par Moreover, if $\big\{x_{i_{1}}, \ldots, x_{i_{r}}\big\}$ is an $x$-cluster in $S_{\sigma, \ell}$ for some $1 \leq \ell \leq n^{(1)}_{\sigma}$, and $\big\{y_{i_{1}}, \ldots, y_{i_{r}}\big\}$ a $y$-cluster in $T_{\sigma, \ell'}$ for some $1 \leq \ell' \leq n^{(2)}_{\sigma}$, then $\Big\{\big(\phi^{1}_{\sigma, \ell}, \phi^{2}_{\sigma, \ell'}\big), (x_{i_{1}}, y_{i_{1}}), \ldots, (x_{i_{r}}, y_{i_{r}})\Big\}$ is winnable for $\big\{S_{\sigma, \ell}, T_{\sigma, \ell'}, k-r\big\}.$

\end{enumerate}

\par We first show that Duplicator can maintain these conditions (via strong induction on $s$). Suppose Duplicator has maintained \ref{same type cond 1} and \ref{same type cond 3} up to round $s$. We call a principal branch (in either tree) \emph{free} if no vertex has been selected from it up to round $s$. Otherwise, we call it \emph{occupied}. For any $\sigma \in \Sigma_{m}$, there exists a free principal branch of type $\sigma$ in $T_{1}|_{m}$ iff there exists a free principal branch of type $\sigma$ in $T_{2}|_{m}$. This is evident from \ref{same type cond 3} and \eqref{same type count} (if \ref{same type cond 3} holds up to round $s$, then the numbers of principal branches of type $\sigma$ that are occupied by round $s$ will be equal in the two trees, for every $\sigma$).
\par Suppose Spoiler, without loss of generality, picks $x_{s+1}$ in round $s+1$. Duplicator's response is split into a few possible cases:
\begin{enumerate}[label={(B\arabic*)},leftmargin=*]
\item \label{same type response 1} If $x_{s+1} = \phi_{1}$, then Duplicator sets $y_{s+1} = \phi_{2}$.
\item \label{same type response 2} Suppose $x_{s+1} \in S_{\sigma, \ell}$ for some $\sigma \in \Sigma_{m}$ and $1 \leq \ell \leq n^{(1)}_{\sigma}$, such that $S_{\sigma, \ell}$ is occupied. Let $\big\{x_{i_{1}}, \ldots, x_{i_{r}}\big\}$ be the $x$-cluster up to round $s$ that belongs to $S_{\sigma, \ell}$. By induction hypothesis \ref{same type cond 3}, there exists some $1 \leq \ell' \leq n^{(2)}_{\sigma}$, such that the $y$-cluster $\big\{y_{i_{1}}, \ldots, y_{i_{r}}\big\} \in T_{\sigma, \ell'}$. Moreover $C = \Big\{\big(\phi^{1}_{\sigma, \ell}, \phi^{2}_{\sigma, \ell'}\big), (x_{i_{1}}, y_{i_{1}}), \ldots, (x_{i_{r}}, y_{i_{r}})\Big\}$ is winnable for $\left\{S_{\sigma, \ell}, T_{\sigma, \ell'}, k-r\right\}.$ By the definition of corresponding vertices, Duplicator finds a corresponding vertex to $x_{s+1}$ in $T_{\sigma, \ell'}$, with respect to $\Big\{S_{\sigma, \ell}, T_{\sigma, \ell'}, k-r, C\Big\}$, and sets it to be $y_{s+1}$. 
\par Note that $C \cup \big\{(x_{s+1}, y_{s+1})\big\}$ is now winnable for $\left\{S_{\sigma, \ell}, T_{\sigma, \ell'}, k-r-1\right\}$, which immediately satisfies \ref{same type cond 3}. 
\item \label{same type response 3} Suppose $x_{s+1} \in S_{\sigma, \ell}$ for some $\sigma \in \Sigma_{m}$ and $1 \leq \ell \leq n^{(1)}_{\sigma}$, such that $S_{\sigma, \ell}$ was free up to round $s$. Duplicator finds an $1 \leq \ell' \leq n^{(2)}_{\sigma}$ such that $T_{\sigma, \ell'}$ was free up to round $s$. By \eqref{ind hyp same type} and the definition of corresponding vertices, Duplicator finds $y_{s+1}$ in $T_{\sigma, \ell'}$ that is a corresponding vertex to $x_{s+1}$ with respect to $\left\{S_{\sigma, \ell}, T_{\sigma, \ell'}, k\right\}$. 
\end{enumerate}

\par It is straightforward to show that Conditions \ref{same type cond 1} and \ref{same type cond 3} imply \ref{DEHR win 1} through \ref{DEHR win 3}. We only show the verification of \ref{DEHR win 2} in the case where $x_{i}$ and $x_{j}$ do not belong to the same principal branch, and neither of them coincides with the root $\phi_{1}$. Suppose then $u_{1}, u_{2}$ be the two distinct children of $\phi_{1}$ such that $x_{i}$ belongs to the principal branch from $u_{1}$ and $x_{j}$ belongs to that from $u_{2}$. By \ref{same type cond 3}, there exist distinct children $v_{1}, v_{2}$ of $\phi_{2}$ such that $y_{i}$ belongs to the principal branch at $v_{2}$ and $y_{j}$ belongs to that at $v_{2}$. Moreover, \ref{same type cond 3} implies that $\rho(x_{i}, v_{1}) = \rho(y_{i}, v_{2})$ and $\rho(x_{j}, v'_{1}) = \rho(y_{j}, v'_{2})$. As the distance between $v_{1}$ and $v'_{1}$, as well as that between $v_{2}$ and $v'_{2}$, is $2$, hence $\rho(x_{i}, x_{j}) = \rho(x_{i}, v_{1}) + \rho(x_{j}, v'_{1}) + 2$, and $\rho(y_{i}, y_{j}) = \rho(y_{i}, v_{2}) + \rho(y_{j}, v'_{2}) + 2$, which gives us \ref{DEHR win 2} for $i, j$. 
\end{proof}

\section{Proof of the main result}\label{proof}
From the discussion of Subsection~\ref{FO_theory_model}, in order to prove Theorem~\ref{main_3}, it suffices to show the following. Fix $k \in \mathbb{N}$ and $\vec{n}$ in the index set $I_{k}$, as in Theorem~\ref{main_3}. Suppose $G_{1}, G_{2}$ are two countable models for the theory $T + \mathcal{A}_{\vec{n}}$. Then Duplicator wins $\EHR\left[G_{1}, G_{2}, k\right]$. 
\par Recall that $x_{i}$ is the vertex selected from $G_{1}$ and $y_{i}$ that from $G_{2}$ in round $i$ of the game. A crucial constituent of constructing a winning strategy for Duplicator will be the judicious selection of a couple of auxiliary vertices $u_{i}$ in $G_{1}$ and $v_{i}$ in $G_{2}$, in round $i$. The vertex $u_{i}$ will be in the same component as $x_{i}$, and $v_{i}$ in the same component as $y_{i}$. Moreover, if $x_{i}$ and $y_{i}$ are in short unicyclic components, then $u_{i}$ and $v_{i}$ are ancestors to $x_{i}$ and $y_{i}$ respectively (including the possibility that $u_{i} = x_{i}$ and $v_{i} = y_{i}$). We next introduce some terminology for ease of exposition of the winning strategy. 
 
\par We call a unicyclic component \emph{long} if its cycle length is more than $2 \cdot 3^{k+3}$, otherwise we call it \emph{short}. For $i, j \in \{1, \ldots, k\}$, we say that $x_{i}, x_{j}$ are \emph{close} if 
\begin{equation}\label{close_cond}
\rho(x_{i}, x_{j}) \leq 2 \cdot 3^{k+2-(i \vee j)},
\end{equation}
else we say that they are \emph{far}. We analogously define $y_{i}, y_{j}$ to be close or far. Suppose $1 \leq i_{1} < \cdots < i_{r} \leq j$ are such that $x_{i_{1}}, \ldots, x_{i_{r}}$ share a common auxiliary vertex $u$, i.e.\ $u_{i_{1}} = \cdots = u_{i_{r}} = u$, but for all other $t$ in $\{1, \ldots, j\}$, we have $u_{t}$ different from $u$. Then we say that $x_{i_{1}}, \ldots, x_{i_{r}}$ form an \emph{$x$-cluster} under $u$ up to round $j$. We analogously define a \emph{$y$-cluster} up to round $j$. If $x_{i}$ (respectively $y_{i}$) belongs to a unicyclic component in $G_{1}$ (respectively $G_{2}$), with 
\begin{equation}\label{deep_cond}
\rho\big(x_{i}, \an(x_{i})\big) \leq 2 \cdot 3^{k+2-i}
\end{equation}
then we say that $x_{i}$ (respectively $y_{i}$) is \emph{located shallow} in this component. We call a short unicyclic component (in either graph) \emph{free} up to round $j$ if $u_{\ell}$ (correspondingly $v_{\ell}$) is not a cycle-vertex of this component for any $1 \leq \ell \leq j$; otherwise, we call it \emph{occupied} by round $j$.


\par Throughout this proof, we maintain the convention set down in Remark~\ref{cycle_convention} while representing unicyclic graphs. 

\par Duplicator maintains the following conditions on the configuration $\left\{(x_{i}, y_{i}): 1 \leq i \leq j\right\}$ as well as the auxiliary pairs $\left\{(u_{i}, v_{i}): 1 \leq i \leq j\right\}$, where $j$ is the number of rounds played so far. 
\begin{enumerate}[label={(Cond \arabic*)},leftmargin=*]

\item \label{cond 1} If $x_{i}$ is located shallow in a short unicyclic component, then $u_{i} = \an(x_{i})$, i.e.\ $u_{i}$ is its cycle-ancestor. Similarly, if $y_{i}$ is located shallow in a unicyclic component, then $v_{i} = \an(y_{i})$.

\item \label{cond 2} For all $i, \ell \in \{1, \ldots, j\}$, $x_{i}, x_{\ell}$ are close iff $y_{i}, y_{\ell}$ are close, and in that case, $\rho(x_{i}, x_{\ell}) = \rho(y_{i}, y_{\ell})$. Furthermore, when these pairs are close, we have $u_{i} = u_{\ell}$ and $v_{i} = v_{\ell}$, except for the scenario where $x_{i}, x_{\ell}$ are located shallow in a common short unicyclic component  but have different cycle-ancestors.

\item \label{cond 3} Suppose $x_{i}$ is neither close to any previously selected $x_{\ell}$ for $1 \leq \ell \leq i-1$ nor located shallow in any short unicyclic component. Then we have $u_{i} = x_{i}$ and $v_{i} = y_{i}$.

\item \label{cond 4} For every $1 \leq i \leq j$, the auxiliary vertex $u_{i}$ equals some cycle-vertex $a_{t}$ of a short unicyclic component $U_{1} = \big(a_{1}, a_{2}, \ldots, a_{s}; A_{1}, A_{2}, \ldots, A_{s}\big)$ in $G_{1}$ iff $v_{i}$ equals $b_{t}$ for some unicyclic component $U_{2} = \big(b_{1}, b_{2}, \ldots, b_{s}; B_{1}, B_{2}, \ldots, B_{s}\big)$ in $G_{2}$. In this case, $U_{1}$ and $U_{2}$ have the same $\big(s, 3^{k+3-\beta}, k\big)$-type, where $\beta$ is the smallest index such that $u_{\beta}$ equals a cycle-vertex of $U_{1}$. Moreover, for $i, \ell \in \{1, \ldots, j\}$, $u_{i}$ and $u_{\ell}$ are cycle-vertices of a common short unicyclic component in $G_{1}$ if and only if $v_{i}$ and $v_{\ell}$ are cycle-vertices of a common short unicyclic component in $G_{2}$.

\item \label{cond 5} For $1 \leq i_{1} < \cdots < i_{r} \leq j$, vertices $x_{i_{1}}, \ldots, x_{i_{r}}$ form an $x$-cluster, up to round $j$, under some vertex $u$ if and only if $y_{i_{1}}, \ldots, y_{i_{r}}$ form a $y$-cluster, up to round $j$, under some vertex $v$. The following also hold, described in two separate cases.
\par If $u$ is a cycle-vertex $a_{t}$ of some short unicyclic component $U_{1} = \big(a_{1}, a_{2}, \ldots, a_{s}; A_{1}, A_{2}, \ldots, A_{s}\big)$, then by \ref{cond 4}, $v$ must be the cycle-vertex $b_{t}$ of a unicyclic component $U_{2} = \big(b_{1}, b_{2}, \ldots, b_{s}; B_{1}, B_{2}, \ldots, B_{s}\big)$, with the same $\big(s, 3^{k+3-\beta}, k\big)$-type, where $\beta$ is the index defined in \ref{cond 4}. Moreover, $\Big\{\big(u, v\big), \big(x_{i_{1}}, y_{i_{1}}\big), \ldots, \big(x_{i_{r}}, y_{i_{r}}\big)\Big\}$ is winnable for $\Big\{A_{t}\big|_{3^{k+3-\beta}}, B_{t}\big|_{3^{k+3-\beta}}, k-r\Big\}$ in this case.
\par If $u$ is not a cycle-vertex of a short unicyclic component, then the trees $B\left(u, 3^{k+2-i_{1}}\right)$ and $B\left(v, 3^{k+2-i_{1}}\right)$ have the same $\left(3^{k+2-i_{1}}, k\right)$-type. Furthermore, $\Big\{\big(u, v\big), \big(x_{i_{1}}, y_{i_{1}}\big), \ldots, \big(x_{i_{r}}, y_{i_{r}}\big)\Big\}$ is winnable for $\Big\{B\left(u, 3^{k+2-i_{1}}\right), B\left(v, 3^{k+2-i_{1}}\right), k-r\Big\}$.

\item \label{cond 6} For every $1 \leq i \leq j$, if $u_{i}$ is a cycle-vertex of a short unicyclic component $U_{1}$, then we have $\rho\big(u_{i}, x_{i}\big) \leq 3^{k+3-\beta} - 3^{k+2-i}$ where $\beta$ is the smallest index for which $u_{\beta}$ is a cycle-vertex of $U_{1}$; otherwise we have $\rho\big(u_{i}, x_{i}\big) \leq 3^{k+2-\beta} - 3^{k+2-i}$ where $\beta$ is the smallest index for which $u_{\beta}$ is the same as $u_{i}$ (notice that $\beta$ can be at most $i$ in either case).

\end{enumerate}

\par Before going into the detailed analysis of how these conditions are maintained throughout the game, let us discuss here some immediate consequences of these conditions, in the following remarks.
\begin{rem}\label{rem_1}
Suppose \ref{cond 5} has been satisfied up to and including round $j$. For any $u$ in $G_{1}$, let $1 \leq i_{1} < \cdots < i_{r} \leq j$, be such that $x_{i_{1}}, \ldots, x_{i_{r}}$ form the $x$-cluster, up to round $j$, under $u$ . Then $y_{i_{1}}, \ldots, y_{i_{r}}$ form the $y$-cluster under some vertex $v$ in $G_{2}$. Moreover, $C = \Big\{(u, v), \big(x_{i_{1}}, y_{i_{1}}\big), \ldots, \big(x_{i_{r}}, y_{i_{r}}\big)\Big\}$ is a winnable configuration for the distance preserving Ehrenfeucht game of $k-r$ rounds on the trees of appropriate depth rooted at $u$ and $v$. By applying \ref{DEHR win 2} to this winnable configuration, we can then conclude that for every $1 \leq t \leq r$, we have
$$\rho\big(x_{i_{t}}, u\big) = \rho\big(y_{i_{t}}, v\big).$$
In particular, for all $1 \leq \ell \leq j$,
\begin{equation}\label{conclusion_rem_1}
\rho\big(u_{\ell}, x_{\ell}\big) = \rho\big(v_{\ell}, y_{\ell}\big).
\end{equation}
We also have, for the same reason, for $t, t' \in \{1, \ldots, r\}$,
\begin{equation}\label{conclusion'_rem_1}
\rho\left(x_{i_{t}}, x_{i_{t'}}\right) = \rho\left(y_{i_{t}}, y_{i_{t'}}\right).
\end{equation}
\end{rem} 

\begin{rem}\label{rem_3}
Fix integers $3 \leq s \leq 2 \cdot 3^{k+3}$ and $0 \leq m \leq 3^{k+3}$, and an $(s, m, k)$-type $\gamma \in \Gamma_{s, m, k}$. Let $m^{1}_{\gamma}$ and $m^{2}_{\gamma}$ be the numbers of unicyclic components with $(s, m, k)$-type $\gamma$, in $G_{1}$ and $G_{2}$ respectively. As $G_{1}$ and $G_{2}$ are models for $T + \mathcal{A}_{\vec{n}}$, for the $\vec{n}$ fixed at the beginning of Section~\ref{proof}, we know that
\begin{equation}\label{same_count_1}
m^{1}_{\gamma} \wedge k = m^{2}_{\gamma} \wedge k = n_{\gamma}.
\end{equation}
Suppose \ref{cond 4} holds for the first $j$ rounds of the game. Let $\mu^{1}_{\gamma}$ and $\mu^{2}_{\gamma}$ respectively denote the numbers of unicyclic components in $G_{1}$ and $G_{2}$, of $(s, m, k)$-type $\gamma$, that have been occupied by round $j$. If $\mu^{1}_{\gamma}$ and $\mu^{2}_{\gamma}$ are not equal, assume without loss of generality that $\mu^{1}_{\gamma} < \mu^{2}_{\gamma}$. Then there must exist some distinct $\ell, \ell' \in \{1, \ldots, j\}$ such that $u_{\ell}$ and $u_{\ell'}$ are cycle-vertices in a common unicyclic component of type $\gamma$ in $G_{1}$, but $v_{\ell}$ and $v_{\ell'}$ are cycle-vertices in two distinct components of type $\gamma$ in $G_{2}$. But this contradicts the last part of \ref{cond 4}. Hence we must have $\mu^{1}_{\gamma} = \mu^{2}_{\gamma}$.
\par The conclusion from this observation and \eqref{same_count_1} is the following: if there exists a unicyclic component of $(s, m, k)$-type $\gamma$ in $G_{1}$ that has been free up to round $j$, then there has to exist a unicyclic component of $(s, m, k)$-type $\gamma$ in $G_{2}$ that has also been free up to round $j$, and vice versa.
\end{rem}

\par We show, via induction on $j$, that the conditions above can indeed be maintained. The base case is the first round. Without loss of generality, let Spoiler select vertex $x_{1}$ from $G_{1}$. The first case is where $x_{1}$ is located shallow in a short unicyclic component $U_{1} = \big(a_{1}, \ldots, a_{s}; A_{1}, \ldots, A_{s}\big)$, with $\an(x_{i}) = a_{t}$ for some $1 \leq t \leq s$. We find $U_{2} = \big(b_{1}, \ldots, b_{s}; B_{1}, \ldots, B_{s}\big)$ in $G_{2}$ that is of the same $\big(s, 3^{k+2}, k\big)$-type as $U_{1}$. We set $u_{1} = a_{t}$ and $v_{1} = b_{t}$, then select $y_{1}$ in $B_{t}$ to be a corresponding vertex to $x_{1}$ with respect to $\left\{A_{t}|_{3^{k+2}}, B_{t}|_{3^{k+2}}, k\right\}$. By \eqref{conclusion_rem_1}, $y_{1}$ is located shallow in $U_{2}$. These choices immediately satisfy \ref{cond 1}, \ref{cond 4}, \ref{cond 5} and \ref{cond 6} (\ref{cond 2} and \ref{cond 3} do not apply).

\par The second case is where $x_{1}$ is not located shallow in any short unicyclic component. In this case, we set $u_{1} = x_{1}$. We find a tree component in $G_{2}$ that has the same type as the tree $B\left(x_{1}, 3^{k+1}\right)$ rooted at $x_{1}$. Then we set the root of this component to be both $y_{1}$ and $v_{1}$. This satisfies \ref{cond 3}, and \ref{cond 5} holds by Lemma \ref{same_tree_type_DEHR}. \ref{cond 1}, \ref{cond 2}, \ref{cond 4} and \ref{cond 6} do not apply here.

\par We now come to the inductive argument. Suppose $j$ rounds have been played so far, and \ref{cond 1} through \ref{cond 5} hold for the current configuration $\big\{(x_{i}, y_{i}): 1 \leq i \leq j\big\}$. Without loss of generality, let Spoiler select $x_{j+1}$ from $G_{1}$ in the $(j+1)$-st round. Duplicator's response needs to be classified into several cases, and these are discussed separately in the subsequent nested subsections. In the analysis of each case, we describe Duplicator's response, and then show that the desired conditions hold for that response.

\subsection{Close move, located shallow in a short unicyclic component:}\label{close,deep} Suppose there exists $1 \leq \alpha \leq j$ such that $x_{j+1}$ is close to $x_{\alpha}$, and $x_{j+1}$ is located shallow in a short unicyclic component $U_{1}$. Let $U_{1} = \big(a_{1}, \ldots, a_{s}; A_{1}, \ldots, A_{s}\big)$, and $\an(x_{j+1}) = a_{t}$ for some $1 \leq t \leq s$. Firstly, observe that this implies that $x_{\alpha}$ is also located shallow in $U_{1}$. The reason is as follows: from \eqref{close_cond} and \eqref{deep_cond}, we have:
\begin{align}
\rho\big(\an(x_{\alpha}), x_{\alpha}\big) &\leq \rho\big(a_{t}, x_{\alpha}\big) \nonumber\\
&\leq \rho\big(\an(x_{j+1}), x_{j+1}\big) + \rho\big(x_{j+1}, x_{\alpha}\big) \nonumber\\
&\leq 2 \cdot 3^{k+1-j} + 2 \cdot 3^{k+1-j} < 2 \cdot 3^{k+2-\alpha},
\end{align}
since $\alpha \leq j$. Then, by induction hypothesis \ref{cond 1}, $u_{\alpha} = \an(x_{\alpha}) = a_{i}$ (say, for some $1 \leq i \leq s$), where $i$ may or may not equal $t$. By induction hypothesis \ref{cond 4}, we know that $v_{\alpha} = b_{i}$ for some short unicyclic component $U_{2} = \big(b_{1}, \ldots, b_{s}; B_{1}, \ldots, B_{s}\big)$, where $U_{1}$ and $U_{2}$ have the same $\big(s, 3^{k+3-\beta}, k\big)$-type, for index $\beta$ as defined in \ref{cond 4}. We set $v_{j+1} = b_{t}$. To select $y_{j+1}$, we have to consider two possibilities, as follows.


\par If $1 \leq i_{1} < \ldots < i_{r} \leq j$ are such that, $x_{i_{1}}, \ldots, x_{i_{r}}$ form the $x$-cluster, up to round $j$, under $a_{t}$, then by induction hypothesis \ref{cond 4} and \ref{cond 5}, we know that $y_{i_{1}}, \ldots, y_{i_{r}}$ form the $y$-cluster, up to round $j$, under $b_{t}$; moreover, $C = \Big\{\big(a_{t}, b_{t}\big), \big(x_{i_{1}}, y_{i_{1}}\big), \ldots, \big(x_{i_{r}}, y_{i_{r}}\big)\Big\}$ is winnable for $\Big\{A_{t}\big|_{3^{k+3-\beta}}, B_{t}\big|_{3^{k+3-\beta}}, k-r\Big\}$. We set $y_{j+1}$ to be a corresponding vertex in $B_{t}$ to $x_{j+1}$ with respect to $\Big\{A_{t}\big|_{3^{k+3-\beta}}, B_{t}\big|_{3^{k+3-\beta}}, k-r, C\Big\}$. By definition of corresponding vertices, $C' = C \cup \Big\{\big(x_{j+1}, y_{j+1}\big)\Big\}$ is winnable for $\Big\{A_{t}\big|_{3^{k+3-\beta}}, B_{t}\big|_{3^{k+3-\beta}}, k-r-1\Big\}$, and this immediately satisfies \ref{cond 4} and \ref{cond 5}. Furthermore, from \eqref{conclusion_rem_1} and the fact that $x_{j+1}$ is located shallow in $U_{1}$, we conclude that $\rho(b_{t}, y_{j+1}) = \rho(a_{t}, x_{j+1}) \leq 2 \cdot 3^{k+1-j}$, hence \ref{cond 6} holds. Also, this tells us that $y_{j+1}$ is located shallow in $U_{2}$. Our choices then tell us that \ref{cond 1} holds.

\par The other possibility is that, there exists no $1 \leq \ell \leq j$ such that $u_{\ell} = a_{t}$. We know by Lemma~\ref{same_tree_type_DEHR} that Duplicator wins $\DEHR\Big[A_{t}\big|_{3^{k+3-\beta}}, B_{t}\big|_{3^{k+3-\beta}}, k, (a_{t}, b_{t})\Big]$ since $A_{t}\big|_{3^{k+3-\beta}}$ and $B_{t}\big|_{3^{k+3-\beta}}$ are of the same $\big(3^{k+3-\beta}, k\big)$-type. So we now select $y_{j+1}$ to be a corresponding vertex to $x_{j+1}$, with respect to $\Big\{A_{t}\big|_{3^{k+3-\beta}}, B_{t}\big|_{3^{k+3-\beta}}, k, \big(a_{t}, b_{t}\big)\Big\}$. Once again, the resulting configuration $C' = \Big\{\big(a_{t}, b_{t}\big), \big(x_{j+1}, y_{j+1}\big)\Big\}$ is winnable for $\Big\{A_{t}\big|_{3^{k+3-\beta}}, B_{t}\big|_{3^{k+3-\beta}}, k-1\Big\}$, thus satisfying \ref{cond 4} and \ref{cond 5}. As above, \ref{cond 6} and \ref{cond 1} hold as well.

\par We are only left to verify \ref{cond 2}, since \ref{cond 3} does not apply to this case. For $i_{1}, \ldots, i_{r}$, \ref{cond 2} is already verified from \eqref{conclusion'_rem_1}. Consider any $1 \leq \ell \leq j$ such that $u_{\ell} = a_{i'}$ with $i'$ different from $t$. By induction hypothesis \ref{cond 4} applied to round $\ell$, we must have $v_{\ell} = \an(y_{\ell}) = b_{i'}$. By \eqref{conclusion_rem_1} applied to round $\ell$, we know that 
\begin{equation}\label{equidistant_3}
\rho\big(a_{i'}, x_{\ell}\big) = \rho\big(b_{i'}, y_{\ell}\big).
\end{equation}
For round $j+1$, due to the same reason, we have
\begin{equation}\label{equidistant_1}
\rho\big(a_{t}, x_{j+1}\big) = \rho\big(b_{t}, y_{j+1}\big).
\end{equation}
Further, we know that $\rho\big(a_{t}, a_{i'}\big) = \rho\big(b_{t}, b_{i'}\big)$, as $U_{1}$ and $U_{2}$ have the same cycle length. Combining this fact with \eqref{equidistant_1} and \eqref{equidistant_3}, we get
\begin{align}
\rho\big(x_{\ell}, x_{j+1}\big) &= \rho\big(x_{\ell}, a_{i'}\big) + \rho\big(a_{i'}, a_{t}\big) + \rho\big(a_{t}, x_{j+1}\big) \nonumber\\
&= \rho\big(y_{\ell}, b_{i'}\big) + \rho\big(b_{i'}, b_{t}\big) + \rho\big(b_{t}, y_{j+1}\big) \nonumber\\
&= \rho\big(y_{\ell}, y_{j+1}\big).
\end{align}
This shows that if $x_{j+1}, x_{\ell}$ are close, then so are $y_{j+1}, y_{\ell}$, and their mutual distances are equal. This verifies \ref{cond 2} for all such pairs $\ell, j+1$. 
\par Consider now some $1 \leq \ell \leq j$ such that $u_{\ell}$ is \emph{not} a cycle-vertex of $U_{1}$. We show that in this case, $x_{\ell}, x_{j+1}$ must be far from each other, and so will be $y_{\ell}, y_{j+1}$. Firstly, if $x_{\ell}, x_{j+1}$ were indeed close, then we would have, from \eqref{deep_cond} applied to round $j+1$, and \eqref{close_cond} applied to the pairs $x_{\ell}, x_{j+1}$, 
\begin{align}
\rho\big(a_{t}, x_{\ell}\big) & \leq \rho\big(a_{t}, x_{j+1}\big) + \rho\big(x_{j+1}, x_{\ell}\big) \nonumber\\
&\leq 2 \cdot 3^{k+1-j} + 2\cdot 3^{k+1-j} < 2 \cdot 3^{k+2-\ell}, \quad \text{as } \ell \leq j.
\end{align}
As $a_{t}$ is a cycle-vertex of $U_{1}$, this means that $x_{\ell}$ must be located shallow in $U_{1}$, and hence $u_{\ell}$ ought to be a cycle-vertex of $U_{1}$ by induction hypothesis \ref{cond 1}. This leads to a contradiction to our initial assumption. Hence $x_{\ell}$ and $x_{j+1}$ must be far from each other. The argument for showing that $y_{j+1}$ and $y_{\ell}$ are far is very similar. This finally completes the verification of \ref{cond 2}.


\subsection{Close move, not located shallow in any short unicyclic component:}\label{close,shallow} Suppose $x_{j+1}$ is close to $x_{\alpha}$ for some $1 \leq \alpha \leq j$, but is not itself located shallow in any short unicyclic component in $G_{1}$. Note that \ref{cond 1} and \ref{cond 3} do not apply here. We set $u_{j+1} = u_{\alpha}$ and $v_{j+1} = v_{\alpha}$, which immediately satisfies the second part of \ref{cond 2}. Suppose $1 \leq i_{1} < \cdots < i_{r} \leq j$ are such that $x_{i_{1}}, \ldots, x_{i_{r}}$ form the $x$-cluster, up to round $j$, under $u_{\alpha}$ ($\alpha$ itself is in the set $\{i_{1}, \ldots, i_{r}\}$). By induction hypothesis \ref{cond 5}, $y_{i_{1}}, \ldots, y_{i_{r}}$ form the $y$-cluster up to round $j$ under $v_{\alpha}$. There are now two possibilities to consider.

\par The first possibility is that $u_{\alpha}$ is a cycle-vertex $a_{t}$ of a short unicyclic component $U_{1} = \big(a_{1}, \ldots, a_{s}; A_{1}, \ldots, A_{s}\big)$. Then by induction hypothesis \ref{cond 4} we know that $v_{\alpha}$ must be the cycle-vertex $b_{t}$ of a unicyclic component $U_{2} = \big(b_{1}, \ldots, b_{s}; B_{1}, \ldots, B{s}\big)$ of $G_{2}$, such that $U_{1}$ and $U_{2}$ have the same $\big(s, 3^{k+3-\beta}, k\big)$-type, where $\beta$ is as defined in \ref{cond 4}. Then $A_{t}$ and $B_{t}$ have the same $\big(3^{k+3-\beta}, k\big)$-type, and by induction hypothesis \ref{cond 5}, the configuration $C = \Big\{\big(u_{\alpha}, v_{\alpha}\big), \big(x_{i_{1}}, y_{i_{1}}\big), \ldots, \big(x_{i_{r}}, y_{i_{r}}\big)\Big\}$ is winnable for $\Big\{A_{t}\big|_{3^{k+3-\beta}}, B_{t}\big|_{3^{k+3-\beta}}, k-r\Big\}$. By induction hypothesis \ref{cond 6}, we know that $\rho\big(x_{\alpha}, u_{\alpha}\big) \leq 3^{k+3-\beta} - 3^{k+2-\alpha}$. Using this inequality and \eqref{close_cond} applied to the pairs $x_{\alpha}, x_{j+1}$, we get:
\begin{align}
\rho\big(u_{\alpha}, x_{j+1}\big) &\leq \rho\big(u_{\alpha}, x_{\alpha}\big) + \rho\big(x_{\alpha}, x_{j+1}\big) \nonumber\\
&\leq 3^{k+3-\beta} - 3^{k+2-\alpha} + 2 \cdot 3^{k+1-j} \nonumber\\
&\leq 3^{k+3-\beta} - 3^{k+2-j} + 2 \cdot 3^{k+1-j}, \quad \text{as } \alpha \leq j, \nonumber\\
&= 3^{k+3-\beta} - 3^{k+1-j},
\end{align}
which at once verifies \ref{cond 6} and shows us that $x_{j+1}$ belongs to $A_{t}|_{3^{k+3-\beta}}$. We can therefore find a corresponding vertex $y_{j+1}$ in $B_{t}|_{3^{k+3-\beta}}$ to $x_{j+1}$ with respect to $\Big\{A_{t}\big|_{3^{k+3-\beta}}, B_{t}\big|_{3^{k+3-\beta}}, k-r, C\Big\}$. By definition of corresponding vertices, we then have $C' = C \cup \big\{(x_{j+1}, y_{j+1})\big\}$ winnable for $\Big\{A_{t}\big|_{3^{k+3-\beta}}, B_{t}\big|_{3^{k+3-\beta}}, k-r-1\Big\}$, and this shows that \ref{cond 5} holds for round $j+1$. From our choice of $u_{j+1}$ and $v_{j+1}$ and the already observed fact that $U_{1}$ and $U_{2}$ have the same $\big(s, 3^{k+3-\beta}, k\big)$-type, we conclude that \ref{cond 4} holds as well.

\par The second possibility is that $u_{\alpha}$ is \emph{not} a cycle-vertex of any short unicyclic component (hence \ref{cond 4} does not apply here). By induction hypothesis \ref{cond 5}, the trees $B\big(u_{\alpha}, 3^{k+2-i_{1}}\big)$ and $B\big(v_{\alpha}, 3^{k+2-i_{1}}\big)$ have the same $\big(3^{k+2-i_{1}}, k\big)$-type, and $C = \Big\{\big(u_{\alpha}, v_{\alpha}\big), \big(x_{i_{1}}, y_{i_{1}}\big), \ldots, \big(x_{i_{r}}, y_{i_{r}}\big)\Big\}$ is winnable for $\Big\{A_{t}\big|_{3^{k+2-i_{1}}}, B_{t}\big|_{3^{k+2-i_{1}}}, k-r\Big\}$. By induction hypothesis \ref{cond 6}, we know that $\rho\big(x_{\alpha}, u_{\alpha}\big) \leq 3^{k+2-i_{1}} - 3^{k+2-\alpha}$. Using this inequality and \eqref{close_cond} applied to the pairs $x_{\alpha}, x_{j+1}$, we get:
\begin{align}
\rho\big(u_{\alpha}, x_{j+1}\big) &\leq \rho\big(x_{\alpha}, u_{\alpha}\big) + \rho\big(x_{\alpha}, x_{j+1}\big) \nonumber\\
&\leq 3^{k+2-i_{1}} - 3^{k+2-\alpha} + 2 \cdot 3^{k+1-j} \nonumber\\
&\leq 3^{k+2-i_{1}} - 3^{k+2-j} + 2 \cdot 3^{k+1-j}, \quad \text{as } \alpha \leq j, \nonumber\\
&= 3^{k+2-i_{1}} - 3^{k+1-j},
\end{align}
which at once verifies \ref{cond 6} and shows us that $x_{j+1}$ belongs to $B\left(u_{\alpha}, 3^{k+2-i_{1}}\right)$. We can therefore find a corresponding vertex $y_{j+1}$ in $B\left(v_{\alpha}, 3^{k+2-i_{1}}\right)$ to $x_{j+1}$ with respect to $\Big\{B\left(u_{\alpha}, 3^{k+2-i_{1}}\right), B\left(v_{\alpha}, 3^{k+2-i_{1}}\right), k-r, C\Big\}$. By definition of corresponding vertices, we then have $C' = C \cup \big\{(x_{j+1}, y_{j+1})\big\}$ winnable for $\Big\{B\left(u_{\alpha}, 3^{k+2-i_{1}}\right), B\left(v_{\alpha}, 3^{k+2-i_{1}}\right), k-r-1\Big\}$, and this shows that \ref{cond 5} holds for round $j+1$.  

\par We are just left to verify \ref{cond 2} for both the above possibilities. For both the possibilities, by \eqref{conclusion'_rem_1} we have
\begin{equation}\label{equidistant_4_general}
\rho\left(x_{j+1}, x_{i_{t}}\right) = \rho\left(y_{j+1}, y_{i_{t}}\right), \quad \text{for all } 1 \leq t \leq r. 
\end{equation}
In particular, we get (since $\alpha \in \{i_{1}, \ldots, i_{r}\}$):
\begin{equation}\label{equidistant_4}
\rho\left(x_{j+1}, x_{\alpha}\right) = \rho\left(y_{j+1}, y_{\alpha}\right). 
\end{equation}
We next show that for all $1 \leq \ell \leq j$ such that $\ell \notin \{i_{1}, \ldots, i_{r}\}$, the pairs $x_{j+1}, x_{\ell}$ and $y_{j+1}, y_{\ell}$ are far. We do this in two separate cases.

\par Suppose $u_{\alpha}$ is a cycle-vertex $a_{t}$ of a short unicyclic component $U_{1} = \big(a_{1}, \ldots, a_{s}; A_{1}, \ldots, A_{s}\big)$. We have already observed above that $x_{j+1}$ is inside tree $A_{t}|_{3^{k+3-\beta}}$, but since it is not located shallow in $U_{1}$, we have $\rho(x_{j+1}, a_{t}) > 2 \cdot 3^{k+1-j}$. We first consider $\ell$ such that $\ell \notin \{i_{1}, \ldots, i_{r}\}$ but $x_{\ell}$ is close to $x_{\alpha}$. Since $x_{\ell}$ does not belong to the $x$-cluster under $u_{\alpha}$, we have $u_{\ell} \neq u_{\alpha}$. Induction hypothesis \ref{cond 2} and our assumption that $x_{\ell}$ and $x_{\alpha}$ are close together imply that $u_{\ell}$ equals $a_{t'}$ for some $1 \leq t' \leq s$ that is distinct from $t$; moreover, $x_{\ell}$ and $x_{\alpha}$ are both located shallow in $U_{1}$. By induction hypothesis \ref{cond 4}, we know that $v_{\alpha} = b_{t}$ and $v_{\ell} = b_{t'}$ for some short unicyclic component $U_{2} = \big(b_{1}, \ldots, b_{s}; B_{1}, \ldots, B_{s}\big)$ in $G_{2}$ with same $\left(s, 3^{k+3-\beta}, k\right)$-type. By \eqref{conclusion_rem_1}, we know that $\rho(y_{\ell}, b_{t'}) = \rho(x_{\ell}, a_{t'})$. 
\par Now, we have selected $y_{j+1}$ as described above, inside $B_{t}|_{3^{k+3-\beta}}$. By \eqref{conclusion_rem_1}, we know that $\rho(y_{j+1}, b_{t}) = \rho(x_{j+1}, a_{t})$. Also, $\rho(a_{t}, a_{t'}) = \rho(b_{t}, b_{t'})$ as $U_{1}$ and $U_{2}$ have the same cycle length. Thus
\begin{align}
\rho(x_{j+1}, x_{\ell}) &= \rho(x_{j+1}, a_{t}) + \rho(a_{t}, a_{t'}) + \rho(a_{t'}, x_{\ell}) \nonumber\\
&= \rho(y_{j+1}, b_{t}) + \rho(b_{t}, b_{t'}) + \rho(b_{t'}, y_{\ell}) = \rho(y_{j+1}, y_{\ell}), \nonumber
\end{align}
and both are greater than $2 \cdot 3^{k+1-j}$ since $x_{j+1}$ is not located shallow inside $U_{1}$ (hence $\rho(x_{j+1}, a_{t}) > 2 \cdot 3^{k+1-j}$).
\par We next consider $\ell$ not in $\{i_{1}, \ldots, i_{r}\}$, and such that $x_{\ell}$ is not close to $x_{\alpha}$. By induction hypothesis \ref{cond 2}, $y_{\ell}$ and $y_{\alpha}$ are far as well. Since $x_{j+1}$ and $x_{\alpha}$ are close, we use triangle inequality to observe that
\begin{align}
\rho(x_{j+1}, x_{\ell}) &\geq \rho(x_{\ell}, x_{\alpha}) - \rho(x_{\alpha}, x_{j+1}) \nonumber\\
&\geq 2 \cdot 3^{k+2-(\alpha \vee \ell)} - 2 \cdot 3^{k+1-j} > 2 \cdot 3^{k+1-j}, \nonumber
\end{align}
hence showing that $x_{j+1}$ and $x_{\ell}$ are far from each other. We show that $y_{j+1}$ and $y_{\ell}$ are far by applying \eqref{equidistant_4} and then using a very similar argument. 

\par Suppose $u_{\alpha}$ is not a cycle-vertex of any short unicyclic component in $G_{1}$. In this case, for any $\ell \notin \{i_{1}, \ldots, i_{r}\}$, since $u_{\ell} \neq u_{\alpha}$, by induction hypothesis \ref{cond 2}, $x_{\ell}$ must be far from $x_{\alpha}$. So, the arguments for showing that $x_{j+1}$ is far from $x_{\ell}$ and $y_{j+1}$ from $y_{\ell}$, are exactly the same as the last part of the above case.

\subsection{Far move, located shallow in a short unicyclic component:}\label{far,deep} Suppose $x_{j+1}$ is far from $x_{\ell}$ for every $1 \leq \ell \leq j$, and located shallow in a short unicyclic component $U_{1}$ of $G_{1}$. Let $U_{1} = \big(a_{1}, \ldots, a_{s}; A_{1}, \ldots, A_{s}\big)$ with $\an(x_{j+1}) = a_{i}$. Note that \ref{cond 3} does not apply here. We firstly set $u_{j+1} = a_{i}$. Then $\rho\big(u_{j+1}, x_{j+1}\big) \leq 2 \cdot 3^{k+1-j}$, by \eqref{deep_cond}. This shows that \ref{cond 6} holds. Next, there are two possible scenarios to consider. 
\par The first is that there exists at least one $1 \leq \ell \leq j$ such that $u_{\ell}$ is a cycle-vertex of $U_{1}$. Let $\beta$ be the smallest such index. By induction hypothesis \ref{cond 4}, we know that there exists a unicyclic component $U_{2} = \big(b_{1}, \ldots, b_{s}; B_{1}, \ldots, B_{s}\big)$ in $G_{2}$ such that $v_{\beta}$ is a cycle-vertex of $U_{2}$, and $U_{1}, U_{2}$ have the same $\big(s, 3^{k+3-\beta}, k\big)$-type. We now set $v_{j+1} = b_{i}$.
\par If there exist indices $1 \leq i_{1} < \cdots < i_{r} \leq j$ such that $x_{i_{1}}, \ldots, x_{i_{r}}$ form the $x$-cluster, up to round $j$, under $a_{i}$, then, by induction hypothesis \ref{cond 5}, $y_{i_{1}}, \ldots, y_{i_{r}}$ form the $y$-cluster, up to round $j$, under $b_{i}$. Moreover, $C = \Big\{\big(a_{i}, b_{i}\big), \big(x_{i_{1}}, y_{i_{1}}\big), \ldots, \big(x_{i_{r}}, y_{i_{r}}\big)\Big\}$ is winnable for $\Big\{A_{i}\big|_{3^{k+3-\beta}}, B_{i}\big|_{3^{k+3-\beta}}, k-r\Big\}$. We then set $y_{j+1}$ in $B_{i}$ to be a corresponding vertex to $x_{j+1}$ with respect to $\Big\{A_{i}\big|_{3^{k+3-\beta}}, B_{i}\big|_{3^{k+3-\beta}}, k-r, C\Big\}$. If the $x$-cluster under $a_{i}$ up to round $j$ is empty, we still can choose $y_{j+1}$ as a corresponding vertex to $x_{j+1}$ by Lemma~\ref{same_tree_type_DEHR},since $A_{i}$ and $B_{i}$ have the same $\big(3^{k+3-\beta}, k\big)$-type. 
\par In either case, by definition of corresponding vertices, we know that $C' = C \cup \big\{(x_{j+1}, y_{j+1})\big\}$ is winnable for $\Big\{A_{i}\big|_{3^{k+3-\beta}}, B_{i}\big|_{3^{k+3-\beta}}, k-r-1\Big\}$, thus satisfying \ref{cond 5} for round $j+1$. We also have immediate satisfiability of \ref{cond 4}. From \eqref{conclusion_rem_1} and the fact that $x_{j+1}$ is located shallow in $U_{1}$, we know that $\rho\big(y_{j+1}, b_{i}\big) = \rho\big(x_{j+1}, a_{i}\big) \leq 2 \cdot 3^{k+1-j}$, hence showing that $y_{j+1}$ is also located shallow in $U_{2}$. Our choices of $u_{j+1}$ and $v_{j+1}$ then validate \ref{cond 1} for round $j+1$. 
\par In order to verify \ref{cond 2}, it is enough to show that for all $1 \leq \ell \leq j$, $y_{\ell}$ and $y_{j+1}$ are far. For all $1 \leq t \leq r$, by \eqref{conclusion'_rem_1}, we conclude that $\rho\big(y_{j+1}, y_{i_{t}}\big) = \rho\big(x_{j+1}, x_{i_{t}}\big)$; hence $y_{j+1}, y_{i_{t}}$ are far from each other as $x_{j+1}, x_{i_{t}}$ are far from each other. Consider $\ell \in \{1, \ldots, j\} \setminus \{i_{1}, \ldots, i_{r}\}$ such that $u_{\ell}$ is a cycle-vertex of $U_{1}$. Then $u_{\ell} = a_{i'}$ for some $1 \leq i' \leq s$ with $i' \neq i$. By induction hypothesis \ref{cond 4}, $v_{\ell} = b_{i'}$. By \eqref{conclusion_rem_1} applied to both rounds $\ell$ and $j+1$, we have $\rho\big(x_{\ell}, a_{i'}\big) = \rho\big(y_{\ell}, b_{i'}\big)$ and $\rho\big(x_{j+1}, a_{i}\big) = \rho\big(y_{j+1}, b_{i}\big)$. We also know that $\rho\big(a_{i}, a_{i'}\big) = \rho\big(b_{i}, b_{i'}\big)$ as $U_{1}$ and $U_{2}$ have the same cycle length. Combining these, we get:
\begin{align}
\rho\big(x_{\ell}, x_{j+1}\big) &= \rho\big(x_{\ell}, a_{i'}\big) + \rho\big(a_{i'}, a_{i}\big) + \rho\big(a_{i}, x_{j+1}\big) \nonumber\\
&= \rho\big(y_{\ell}, b_{i'}\big) + \rho\big(b_{i'}, b_{i}\big) + \rho\big(b_{i}, y_{j+1}\big) \nonumber\\
&= \rho\big(y_{\ell}, y_{j+1}\big).
\end{align}
This again shows that $y_{\ell}, y_{j+1}$ are far since $x_{\ell}, x_{j+1}$ are far. Finally, consider $1 \leq \ell \leq j$ such that $u_{\ell}$ is not a cycle-vertex of $U_{1}$. If $y_{\ell}$ belongs to a component different from $U_{2}$, nothing left to prove. Assume that $y_{\ell} \in U_{2}$. If $y_{\ell}$ and $y_{j+1}$ were close, then using \eqref{close_cond} applied to this pair, \eqref{conclusion_rem_1} to $j+1$,and the fact that $x_{j+1}$ is located shallow in $U_{1}$, we get:
\begin{align}\label{verify_cond_2_far_deep}
\rho\big(y_{\ell}, b_{i}\big) &\leq \rho\big(y_{\ell}, y_{j+1}\big) + \rho\big(y_{j+1}, b_{i}\big) \nonumber\\
&\leq 2 \cdot 3^{k+1-j} + \rho\big(x_{j+1}, a_{i}\big) \nonumber\\
&\leq 2 \cdot 3^{k+1-j} + 2 \cdot 3^{k+1-j} < 2 \cdot 3^{k+2-\ell}, \quad \text{as } \ell \leq j.
\end{align}
As $b_{i}$ is a cycle-vertex of $U_{2}$, this implies that $y_{\ell}$ is located shallow in $U_{2}$. Induction hypothesis \ref{cond 1} tells us that $v_{\ell} = \an(y_{\ell})$ in $U_{2}$, and induction hypothesis \ref{cond 4} then tells us that $u_{\ell}$ must be a cycle-vertex of $U_{1}$, which contradicts the assumption we started with. This completes the verification of \ref{cond 2}.


\par The second possible scenario is that there exists no $1 \leq \ell \leq j$ such that $u_{\ell}$ is a cycle-vertex of $U_{1}$, i.e.\ $U_{1}$ was free up to round $j$. Suppose $U_{1} = \big(a_{1}, \ldots, a_{s}; A_{1}, \ldots, A_{s}\big)$, with $\an(x_{j+1}) = a_{i}$. We first set $u_{j+1} = a_{i}$. By Remark~\ref{rem_3}, we find a unicyclic component $U_{2} = \big(b_{1}, \ldots, b_{s}; B_{1}, \ldots, B_{s}\big)$ in $G_{2}$ such that $U_{2}$ has the same $\big(s, 3^{k+2-j}, k\big)$-type as $U_{1}$, and was free up to round $j$. By Lemma~\ref{same_tree_type_DEHR}, and the fact that $A_{i}, B_{i}$ have the same $\big(3^{k+2-j}, k\big)$-type, we choose $y_{j+1}$ to be a corresponding vertex to $x_{j+1}$ in $B_{i}$ with respect to $\Big\{A_{i}\big|_{3^{k+2-j}}, B_{i}\big|_{3^{k+2-j}}, k, (a_{i}, b_{i})\Big\}$. We also set $v_{j+1} = b_{i}$. 
\par By the fact that $x_{j+1}$ is located shallow in $U_{1}$, \eqref{deep_cond} and \eqref{conclusion_rem_1}, we conclude that $\rho\big(y_{j+1}, b_{i}\big) = \rho\big(x_{j+1}, a_{i}\big) \leq 2 \cdot 3^{k+1-j}$; hence $y_{j+1}$ is located shallow in $U_{2}$. Our choices of $u_{j+1}$ and $v_{j+1}$ then validate \ref{cond 1}. Note that $U_{1}$ being free up to round $j$, the smallest index $\beta$ such that $u_{\beta}$ is a cycle-vertex of $U_{1}$ is actually $j+1$, and we indeed have $U_{1}$ and $U_{2}$ of the same $\big(s, 3^{k+2-j}, k\big)$-type. This, along with our choices of $u_{j+1}$ and $v_{j+1}$ and the fact that $U_{1}$ and $U_{2}$ were both free up to round $j$, validate \ref{cond 4}. By \eqref{deep_cond}, we have $\rho\big(x_{j+1}, a_{i}\big) \leq 2 \cdot 3^{k+1-j} = 3^{k+2-j} - 3^{k+1-j}$, thus validating \ref{cond 6} for round $j+1$. By definition of corresponding vertices, we know that $C = \Big\{\big(a_{i}, b_{i}\big), \big(x_{j+1}, y_{j+1}\big)\Big\}$ is winnable for $\Big\{A_{i}\big|_{3^{k+2-j}}, B_{i}\big|_{3^{k+2-j}}, k-1\Big\}$, thus validating \ref{cond 5} for round $j+1$. 
\par To verify \ref{cond 2}, we once again just show that $y_{\ell}$ and $y_{j+1}$ are far for every $1 \leq \ell \leq j$. If not, then by derivations similar to that of \eqref{verify_cond_2_far_deep}, we have $\rho\big(y_{\ell}, b_{i}\big) < 2 \cdot 3^{k+2-\ell}$, hence showing that $y_{\ell}$ is located shallow in $U_{2}$. By induction hypothesis \ref{cond 1}, $v_{\ell}$ must then be a cycle-vertex of $U_{2}$, contradicting our choice of $U_{2}$ as free up to round $j$. This completes the verification of \ref{cond 2}. 

\subsection{Far move, not located shallow any short unicyclic component:}\label{far, shallow} Suppose there does not exist any $1 \leq \ell \leq j$ with $x_{\ell}$ and $x_{j+1}$ close, and $x_{j+1}$ is not located shallow in any short unicyclic component. We set $u_{j+1} = x_{j+1}$. Consider the tree $B\left(u_{j+1}, 3^{k+1-j}\right)$ rooted at $u_{j+1}$ up to depth $3^{k+1-j}$. Let the $\big(3^{k+1-j}, k\big)$-type of this tree be $\sigma$. By Theorem \ref{main_1}, each of $G_{1}$ and $G_{2}$ will contain at least $k$ tree components of type $\sigma$. Since less than $k$ rounds have been played so far, we find such a component in $G_{2}$ from which no vertex has been selected up to round $j$, and set its root to be both $v_{j+1}$ and $y_{j+1}$. 
\par Firstly, our choices of $u_{j+1}, v_{j+1}$ and $y_{j+1}$ are consistent with \ref{cond 3}. As $\rho(u_{j+1}, x_{j+1}) = 0$, hence \ref{cond 6} holds as well. \ref{cond 1} and \ref{cond 4} do not apply here. Since no $y_{\ell}$ belongs to the same component as $y_{j+1}$ for $1 \leq \ell \leq j$, hence $y_{\ell}, y_{j+1}$ are far (their distance is infinite). Hence \ref{cond 2} holds immediately. Finally, observe that $u_{j+1} = x_{j+1}$ and $v_{j+1} = y_{j+1}$ are the roots of the trees $B\left(u_{j+1}, 3^{k+1-j}\right)$ and $B\left(v_{j+1}, 3^{k+1-j}\right)$, and these two trees have the same $\big(3^{k+1-j}, k\big)$-type. By Lemma~\ref{same_tree_type_DEHR}, Duplicator wins $\DEHR\Big[B\left(u_{j+1}, 3^{k+1-j}\right), B\left(v_{j+1}, 3^{k+1-j}\right), k, (u_{j+1}, v_{j+1})\Big]$, which verifies \ref{cond 5}.

\par This brings us to the end of the analysis of Duplicator's response to all possible moves by Spoiler. The conditions \ref{cond 1} through \ref{cond 6} are stronger than \ref{EHR 1} and \ref{EHR 2}. Hence Duplicator wins the Ehrenfeucht game $\EHR[G_{1}, G_{2}, k]$.

\subsection{The final conclusion for Theorem~\ref{main_3}:} Note that the above proof lets us conclude that for a fixed positive integer $k$, if we specify that a graph $G$ satisfies properties \ref{1}, \ref{2} and $\mathcal{A}_{\vec{n}}$, for some $\vec{n} \in I_{k}$, then this completely describes FO properties of quantifier depth $\leq k$ that hold for $G$. Given an arbitrary FO property $A$ of quantifier depth at most $k$, we can determine without any uncertainty whether $A$ holds for $G$ or not. Hence, for every $\vec{n} \in I_{k}$, the theory $T + \mathcal{A}_{\vec{n}}$ is a $k$-complete theory. It is also immediate that for two distinct $\vec{n}$ and $\vec{m}$ in $I_{k}$, only one of $\mathcal{A}_{\vec{n}}$ and $\mathcal{A}_{\vec{m}}$ can hold for any graph, hence $T \models \neg\left(\mathcal{A}_{\vec{m}} \wedge \mathcal{A}_{\vec{n}}\right)$. Moreover, $I_{k}$ is exhaustive in the sense that $\bigcup_{\vec{n} \in I_{k}} \mathcal{A}_{\vec{n}}$ is the entire sample space. Thus any countable model that satisfies $T$ will satisfy $A_{\vec{n}}$ for precisely one $\vec{n}$ in $I_{k}$.

\par So finally, to conclude, from the definition given in Section~\ref{FO_theory_model}, that $\left\{\mathcal{A}_{\vec{n}}: \vec{n} \in I_{k}\right\}$ is a complete set of $k$-completions of $T$, it is enough to establish that the limit
\begin{equation}\label{lim_exists}
\lim_{n \rightarrow \infty} P\left[G(n, cn^{-1}) \models \mathcal{A}_{\vec{n}}\right]
\end{equation}
exists. We show this in Section~\ref{lim_probab_sec}. We additionally show that it is not possible to exclude any $\vec{n}$ from $I_{k}$ while constructing this comlete set of $k$-completions, i.e.\
\begin{equation}\label{limiting_probability}
\lim_{n \rightarrow \infty} P\left[G(n, cn^{-1}) \models \mathcal{A}_{\vec{n}}\right] > 0 \quad \text{for all } \vec{n} \in I_{k}.
\end{equation}

\section{Limiting probabilities of the $k$-completions}\label{lim_probab_sec}
We fix positive integers $M_{1}$ and $M_{2}$, and consider $I_{M_{1}, M_{2}}$ to be the set of all sequences $\vec{n} = \left(n_{\gamma}: \gamma \in \Gamma_{s, m, k}, 3 \leq s \leq M_{1}, 0 \leq m \leq M_{2}\right)$ that satisfy \eqref{consistency_cond_truncated}. We show, in subsequent subsections, that the following holds:
\begin{equation}\label{limiting_probability_generalized}
\lim_{n \rightarrow \infty} P\left[G(n, cn^{-1}) \models \mathcal{A}_{\vec{n}}\right] \text{ exists and is positive, for all } \vec{n} \in I_{M_{1}, M_{2}}.
\end{equation}
In Subsection~\ref{no_short_cycles}, we show that the probability of the event that there are no cycles of length $\leq M_{1}$ in $G\left(n, c n^{-1}\right)$ (i.e.\ when $\vec{n} = \vec{0}$) converges to a positive limit as $n \rightarrow \infty$. Using this fact, in Subsection~\ref{general_short_cycle_counts}, we show that \eqref{limiting_probability_generalized} holds for any general $\vec{n}$ in $I_{M_{1}, M_{2}}$. 

\subsection{When no short cycles are present}\label{no_short_cycles}
Consider the event $A$ that exists no cycle of length $i$ for all $3 \leq i \leq M_{1}$. Let $\mathcal{S}_{i}$ denote the set of all subsets of size $i$ of the vertex set $V$ of $G\left(n, cn^{-1}\right)$. For every $S \in \mathcal{S}_{i}$, let $\mathbf{1}_{S}$ be indicator for the event that the induced subgraph on $S$ is a cycle of length $i$. For $3 \leq i \leq M_{1}$ and $S \in \mathcal{S}_{i}$, we have $E[\mathbf{1}_{S}] = \left(\frac{c}{n}\right)^{i} = c^{i} n^{-i}.$ 
\par Consider $S \in \mathcal{S}_{i}$ and $T \in \mathcal{S}_{j}$, for some $i, j \in \{3, \ldots, M_{1}\}$. Let the number of common vertices between $S$ and $T$ be $\ell+1$, where $\ell + 1 \leq i \wedge j$, and $\ell \geq 1$. The number of edges required for both $\mathbf{1}_{S}$ and $\mathbf{1}_{T}$ to be true is at least $i+j-\ell$, whereas the total number of vertices in $S \cup T$ is $i+j-\ell-1$. Thus the cases where the cycles on $S$ and $T$ share exactly $\ell$ edges, have the dominant probabilities, given by $\Theta\left(\left(\frac{c}{n}\right)^{i+j-\ell}\right) = \Theta\left(n^{-i-j+\ell}\right).$
Now, there are ${n \choose \ell+1} \cdot {n-\ell \choose i-\ell-1} \cdot {n-i \choose j-\ell-1} = \Theta\left(n^{i+j-\ell-1}\right)$ ways of choosing $S$ and $T$. Hence 
\begin{equation}
\sum_{\substack{S, T \subseteq V\\|S| = i, |T| = j, |S \cap T| = \ell+1}} E[\mathbf{1}_{S} \mathbf{1}_{T}] = \Theta\left(n^{i+j-\ell-1}\right) \cdot \Theta\left(n^{-i-j+\ell}\right) = \Theta\left(n^{-1}\right).
\end{equation}
Summing over $i, j \leq \{3, \ldots, M_{1}\}$ and $\ell + 1 \leq i \wedge j$, and noting that the sum has finitely many terms, we get 
\begin{equation}\label{Janson_interaction}
\sum_{i, j \in \{3, \ldots, M_{1}\}} \sum_{\ell + 1 \leq \min\{i, j\}} \sum_{\substack{S, T \subseteq V\\|S| = i, |T| = j, |S \cap T| = \ell+1}} E[\mathbf{1}_{S} \mathbf{1}_{T}] = O\left(n^{-1}\right).
\end{equation}
A direct application of Janson's inequality gives us
\begin{align}\label{eq_final_no_short_cycle}
\lim_{n \rightarrow \infty} P\left[\text{no cycle of length} \leq M_{1}\right] &= \lim_{n \rightarrow \infty} \prod_{i=3}^{M_{1}} \left\{1 - c^{i} n^{-i}\right\}^{{n \choose i}} = \prod_{i=3}^{M_{1}} \exp\left\{-\frac{c^{i}}{i!}\right\}.
\end{align}
Notice that this not only establishes \eqref{limiting_probability_generalized}, but at the same time establishes that the limit exists, as required in \eqref{lim_exists}, when $\vec{n} = \vec{0}$. 

\subsection{The general case:}\label{general_short_cycle_counts}
There are two key ideas of the proof in this subsection is the following. Fix $\vec{n}$ in $I_{M_{1}, M_{2}}$. 
\begin{defn}\label{picture}
Given a graph $G$, we define the $(M_{1}, M_{2})$-picture of $G$, denoted $\pic(G) = \pic_{M_{1}, M_{2}}(G)$ (we drop the subscripts whenever their values are obvious from the context), gives an \emph{exact} description of all the short unicyclic components up to depth $M_{2}$, i.e.\ for every unicyclic component $U$ whose cycle length is $\leq M_{1}$, we know the exact structure of $U|_{M_{2}}$. In particular, we do \emph{not} consider the cutoff $k$ when we describe $\pic(G)$. 
\end{defn}
To give an example, suppose the graph has a triangle with two vertices childless and one having exactly $k+1$ children that are all childless. For $M_{2} \geq 2$, the description of the triangle in $\pic(G)$ should not simply state that two of its vertices are childless and one has at least $k$ children that are all childless, but rather state the exact counts. Let us denote by $\mathcal{P} = \mathcal{P}_{M_{1}, M_{2}}$ the set of all possible $(M_{1}, M_{2})$-pictures . 

\begin{defn}\label{picture_compatible}
Given $\vec{n} \in I_{M_{1}, M_{2}}$ and a graph $G$, we call $\pic_{M_{1}, M_{2}}(G)$ \emph{compatible} with $\vec{n}$ if $G \models \mathcal{A}_{\vec{n}}$ (or equivalently, we can say that $\pic(G) \models \mathcal{A}_{\vec{n}}$).
\end{defn}
Let $\mathcal{P}(\vec{n}) = \mathcal{P}_{M_{1}, M_{2}}(\vec{n})$ denote the set of all pictures that are compatible with $\vec{n}$. Note that for many $\vec{n}$, the set $\mathcal{P}(\vec{n})$ is infinite. 

\par W fix a picture $H$ from $\mathcal{P}_{M_{1}, M_{2}}$, and estimate below the probability that $\pic\big(G\left(n, c n^{-1}\right)\big) = H$. Let $M_{H}$ denote the number of vertices in the picture $H$ (crucially, $M_{H}$ does not depend on $n$ for a fixed $H$). Since $H$ consists of unicyclic components, hence the number of edges in $H$ will also be $M_{H}$. We choose a subset $S$ of $M_{H}$ vertices from $G(n, c n^{-1})$ in ${n \choose M_{H}} \sim n^{M_{H}}/M_{H}!$ ways. The probability that the induced subgraph on $S$ will be the picture $H$ is given by 
\begin{equation}
C_{H} \cdot p^{M_{H}} (1-p)^{{M_{H} \choose 2} - M_{H}} = C_{H} \cdot \frac{c^{M_{H}}}{n^{M_{H}}} \cdot \left\{1 - \frac{c}{n}\right\}^{{M_{H} \choose 2} - M_{H}} \sim C_{H} \cdot \frac{c^{M_{H}}}{n^{M_{H}}}.
\end{equation}
The constant $C_{H}$ is derived from the number of automorphisms of $H$, and is independent of $n$. The only edges that can exist between the subgraph on $S$ and that on $V \setminus S$ must be between $\mathcal{L}$ and $V \setminus S$, where $\mathcal{L}$ is the set of leaves in $H$ at depth $M_{2}$. The cardinality of $\mathcal{L}$ depends only on $H$ and not on $n$, and let this number be $L_{H}$. The probability that there is no edge between $S \setminus \mathcal{L}$ and $V \setminus S$ is given by
\begin{equation}
\left(1 - \frac{c}{n}\right)^{\left(M_{H} - L_{H}\right) \left(n - M_{H}\right)} \sim e^{-c\left(M_{H} - L_{H}\right)}.
\end{equation}
Finally, the subgraph induced on $\mathcal{L} \cup \left(V \setminus S\right)$ is distributed the same as $G\left(n - M_{H} + L_{H}, c n^{-1}\right)$, and it must not contain any short unicyclic components. By repeating the computations as in Subsection~\ref{no_short_cycles} and noting that $n \sim n - M_{H} + L_{H}$, we get:
\begin{multline} 
\pr\left[\text{subgraph induced on } (V \setminus S) \cup \mathcal{L} \text{ has no short cycles}\right] \sim \prod_{i=3}^{M_{1}} \exp\left\{-\frac{c^{i}}{i!}\right\}.
\end{multline}
Hence finally, we have:
\begin{align}\label{limit_each_H}
&\pr\left[\pic\big(G\left(n, c n^{-1}\right)\big) = H\right] \nonumber\\
\sim & \frac{n^{M_{H}}}{M_{H}!} \cdot C_{H} \cdot \frac{c^{M_{H}}}{n^{M_{H}}} \cdot e^{-c\left(M_{H} - L_{H}\right)} \cdot \prod_{i=3}^{M_{1}} \exp\left\{-\frac{c^{i}}{i!}\right\} \nonumber\\
=& \frac{C_{H} \cdot c^{M_{H}}}{M_{H}!} \cdot e^{-c\left(M_{H} - L_{H}\right)} \cdot \prod_{i=3}^{M_{1}} \exp\left\{-\frac{c^{i}}{i!}\right\},
\end{align}
thus showing that the limit of $\pr\left[\pic\big(G\left(n, c n^{-1}\right)\big) = H\right]$ exists as $n \rightarrow \infty$ and it is positive. 

\par So far, we have only considered the limit for every fixed picture $H$. Our goal is to show that for every $\vec{n} \in I_{M_{1}, M_{2}}$, the limit of $\pr\left[G(n, c n^{-1}) \models \mathcal{A}_{\vec{n}}\right]$, or equivalently, $\pr\left[\pic\big(G\left(n, c n^{-1}\right)\big) \in \mathcal{P}(\vec{n})\right]$, exists as $n \rightarrow \infty$. The crucial observation is that an interchange of limit and summation over all $H \in \mathcal{P}(\vec{n})$ is not allowed when $\mathcal{P}(\vec{n})$ is infinite.

\par To this end, for any $N \in \mathbb{N}$, we split $\mathcal{P}(\vec{n})$ into two parts: the set $\mathcal{P}_{N}(\vec{n})$ of pictures $H$ with $M_{H} \leq N$, and the set $\mathcal{P}_{> N}(\vec{n})$ of pictures $H$ with $M_{H} > N$. Clearly, $\mathcal{P}_{N}(\vec{n})$ is a finite set, and hence for every $N$, from \eqref{limit_each_H}, we have
\begin{equation}\label{limit_pics_upto_N}
\lim_{n \rightarrow \infty} \pr\left[\pic\big(G\left(n, c n^{-1}\right)\big) \in \mathcal{P}_{N}(\vec{n})\right] = \sum_{H \in \mathcal{P}_{N}(\vec{n})} \frac{C_{H} \cdot c^{M_{H}}}{M_{H}!} \cdot e^{-c\left(M_{H} - L_{H}\right)} \cdot \prod_{i=3}^{M_{1}} \exp\left\{-\frac{c^{i}}{i!}\right\}.
\end{equation}

\par There are essentially two ways (or a combination of both) that one may end up with a picture that has too many vertices. These are as follows:
\begin{enumerate}
\item either there are many short unicyclic components;
\item or there exists at least one vertex inside the picture that has a very high degree.
\end{enumerate} 

\par Consider $\many_{W}$ to be the event that the total number of short unicyclic components is bigger than $W M_{1}$, for any given positive integer $W$. By pigeon-hole principle, there must exist some $3 \leq s \leq M_{1}$ such that the number of unicyclic components that have cycle length $s$ exceeds $W$. Let $\many_{W, s}$ denote the event that there are at least $W$ many cycles of length $s$ (and the components containing these cycles are going to be disjoint from each other, since from \ref{1} of Theorem \ref{main_1}, we know that almost surely there is no bicyclic component). We can choose a subset of $s$ vertices in ${n \choose s} \sim n^{s}/s!$ many ways, and the probability that the induced subgraph on this subset is a cycle of length $s$ is given by $\frac{(s-1)!}{2} \cdot \frac{c^{s}}{n^{s}} \cdot \left\{1 - \frac{c}{n}\right\}^{{s \choose 2} - s} \sim \frac{(s-1)!}{2} \cdot \frac{c^{s}}{n^{s}}.$ Hence the expected number of cycles of length $s$ in $G(n, c n^{-1})$ is asymptotically $\sim \frac{c^{s}}{2s} = O(1)$. Using Markov's inequality, we then get:
\begin{align}
&\pr\left[\#\text{ cycles of length } s \text{ in } G(n, c n^{-1}) > W\right] \nonumber\\
\leq & \E\left[\#\text{ cycles of length } s \text{ in } G(n, c n^{-1})\right] W^{-1} \sim \frac{c^{s}}{2s} \cdot W^{-1}, \quad \text{as } n \rightarrow \infty,
\end{align}
showing that the limit goes to $0$ as $W \rightarrow \infty$.
\par Next, consider the event $\bu_{W}$ that there exists some node $v$ within the $(M_{1}, M_{2})$-picture of $G(n, c n^{-1})$ such that $v$ has more than $W$ many neighbours. By definition, this means that there exists a short cycle from which $v$ is at a distance $\leq M_{2}$. Define $\bu_{s, d, W}$ to be the event that there exists a cycle of length $s$, and a node $v$ at distance $d$ from the cycle, such that $v$ has degree more than $W$. 

\par Call a graph an \emph{$(s, d, W)$-key} if it is a connected graph that consists of precisely the following: a cycle of length $s$, a node $v$ not on the cycle, a path of length $d$ between $v$ and a vertex $u$ on the cycle, and $W$ vertices not on the path nor on the cycle that are adjacent to $v$. For $\bu_{s, d, W}$ to hold, at least one $(s, d, W)$-key must be present in $G(n, c n^{-1})$. 

\par We estimate the expected number of $(s, d, W)$-keys in $G(n, c n^{-1})$ here. First, we have to choose $s$ many vertices to form the cycle, which we can do in ${n \choose s} \sim n^{s}/s!$ many ways. They can be arranged to form a cycle in $(s-1)!/2$ many ways. We can choose the vertex $u$ on the cycle in ${s \choose 1} = s$ many ways. We can choose the $d$ many vertices (including $v$) on the path between $u$ and $v$ in ${n-s \choose d} \sim n^{d}/d!$ many ways, and arrange them in $d!$ many ways. Finally, we can choose the $W$ remaining neighbours of $v$ in ${n-s-d \choose W} \sim n^{W}/W!$ many ways. Next, we note that there are $s$ many edges in the cycle, $d$ many edges along the path, and $W$ many edges between $v$ and its $W$ neighbours (not on the path between $v$ to $u$) -- hence a total of $s+d+W$ many edges.

\par Thus the expected number of $(s, d, W)$-keys in $G(n, c n^{-1})$ is
\begin{align} 
&\frac{n^{s}}{s!} \cdot \frac{(s-1)!}{2} \cdot s \cdot \frac{n^{d}}{d!} \cdot d! \cdot \frac{n^{W}}{W!} \cdot \left(\frac{c}{n}\right)^{s + d + W} = \frac{c^{s+d+W}}{2 W!}.
\end{align}
Consequently, again by an application of Markov's inequality:
\begin{align}
\pr\left[\bu_{W} \text{ holds for } G(n, c n^{-1})\right] &= \pr\left[\bigcup_{\substack{3 \leq s \leq M_{1}\\0 \leq d \leq M_{2}}} \left\{\bu_{s, d, W} \text{ holds for } G(n, c n^{-1})\right\}\right] \nonumber\\
&\sim \sum_{\substack{3 \leq s \leq M_{1}\\0 \leq d \leq M_{2}}} \frac{c^{s+d+W}}{2 W!}, \quad \text{as } n \rightarrow \infty,
\end{align}
and this limit of the probability clearly goes to $0$ as $W \rightarrow \infty$. 

\par These estimates show us that both $\many_{W}$ and $\bu_{W}$ occur with $o(1)$ probability as $n \rightarrow \infty$. Given an arbitrarily small but fixed $\epsilon$, there exists $W_{0} \in \mathbb{N}$ such that, for all $W \geq W_{0}$,
\begin{equation}\label{ep_1}
\lim_{n \rightarrow \infty}\pr\left[\many_{W} \text{ holds for } G(n, c n^{-1})\right] \leq \frac{\epsilon}{2},
\end{equation}
and
\begin{equation}\label{ep_2}
\lim_{n \rightarrow \infty}\pr\left[\bu_{W} \text{ holds for } G(n, c n^{-1})\right] \leq \frac{\epsilon}{2}.
\end{equation}
We find $N_{0} \in \mathbb{N}$ such that if the $(M_{1}, M_{2})$-picture of a graph $G$ contains more than $N_{0}$ many vertices, then either $\many_{W_{0}}$ or $\bu_{W_{0}}$ or both must hold for $G$. By \eqref{ep_1} and \eqref{ep_2},
\begin{align}\label{limsup_>N}
& \limsup_{n \rightarrow \infty} \pr\left[\pic\big(G(n, c n^{-1})\big) \in \mathcal{P}_{> N_{0}}(\vec{n})\right] \leq \epsilon.
\end{align}
Since $\epsilon$ is arbitrary, this lets us conclude that the limit of $\pr\left[\pic\big(G\left(n, c n^{-1}\right)\big) \in \mathcal{P}(\vec{n})\right]$ exists as $n \rightarrow \infty$. Moreover, from \eqref{limit_each_H}, we conclude that this limit is positive for each $\vec{n} \in I_{M_{1}, M_{2}}$.

This completes the proof that indeed $\left\{\mathcal{A}_{\vec{n}}: \vec{n} \in I_{k}\right\}$ forms a complete set of $k$-completions for $T$.

\section{Acknowledgements}
The author humbly thanks her doctoral advisor Prof.\ Joel Spencer for suggesting the problem addressed in this paper to her, and for sharing his extremely helpful thoughts with her.

\end{document}